\documentclass[3p,times]{elsarticle}
\usepackage[toc,page,title,titletoc,header]{appendix}
\usepackage{amsmath}
\usepackage{amssymb}
\usepackage{amsthm}
\usepackage{enumerate}
\usepackage{enumitem}
\usepackage{mathrsfs}
\usepackage{graphicx}
\usepackage{epsfig}
\usepackage{tikz,pgfplots,tikz-3dplot}
\usepackage{epstopdf}
\usepackage{microtype}
\usetikzlibrary{fit}
\biboptions{sort&compress}

\newtheorem{thm}{Theorem}[section]
\newtheorem{lem}[thm]{Lemma}
\newtheorem{rem}{Remark}
\renewcommand{}

\newcommand{\ttau}{\Delta t}
\newcommand{\mS}{\mathcal{S}}

\def\epsilon{\varepsilon} 
\newcommand{\mat}[1]{\boldsymbol{#1}}

\allowdisplaybreaks
\begin{document}
\begin{frontmatter}
\title{
Willmore regularized sharp-interface model for strongly anisotropic solid-state dewetting with axisymmetric geometry: modeling and simulation}

\author[1]{Meng Li}
\address[1]{School of Mathematics and Statistics, Zhengzhou University,
Zhengzhou 450001, China.}
\ead{This author's research was supported by National Natural Science Foundation of China (Nos. 11801527, U23A2065), the China Postdoctoral Science Foundation (No. 2023T160589).
Corresponding author: limeng@zzu.edu.cn. }
\author[1]{Chunjie Zhou}
%\ead{519018323@qq.com}

\begin{abstract}
In this work, we consider the three-dimensional solid-state dewetting with strongly anisotropic surface energy, assuming an axisymmetric morphology of the thin film. 
However, when surface energy exhibits strong anisotropy, certain orientations may be missing from the equilibrium shapes, which will lead to an ill-posed governing equation. 
By incorporating the Willmore energy, we define a regularized total free energy and rigorously derive a sharp-interface model based on thermodynamic variations. 
We further develop a numerical scheme for the sharp-interface model that can preserve two important structural properties, including both the volume-conservation and energy-stability laws. We conclude by presenting a series of numerical simulations that illustrate the accuracy and structure-preserving properties. 
More importantly, extensive numerical simulations clearly demonstrate that our schemes can significantly enhance mesh quality, which is beneficial for long-term computations.

\end{abstract}
%%%%% end %%%%%%%%%%%

%%%%% Keywords %%%%%%%%%%%


%\begin{AMS}
%74H15, 74S05, 74M15, 65Z99
%\end{AMS}

\end{frontmatter}

\pagestyle{myheadings} \markboth{~}
{}

\section{Introduction}\label{sec1}
% Introduce the SSD as your first and second papers. Background

Solid-state dewetting (SSD), a widely observed phenomenon in physics and materials science that occurs in solid–solid–vapor systems, could be used to describe the agglomerative process of solid thin films on a substrate.
A solid film adhered to the substrate is inherently unstable or metastable in its as-deposited state due to the influences of surface tension and capillarity. 
This instability can give rise to complex morphological evolutions, such as fingering instabilities \cite{Kan05,Ye10b,Ye11a,Ye11b}, edge retraction \cite{Wong00,dornel2006surface,hyun2013quantitative}, faceting \cite{Jiran90,Jiran92,Ye10a} and pinch-off events \cite{Jiang12,Kim15}.
SSD has found widespread applications in a variety of modern technologies \cite{Mizsei93,Armelao06,Schmidt09}. This broad range of applications has generated considerable interest and motivated extensive efforts to explore and understand its underlying mechanisms, 
encompassing experimental investigations \cite{Jiran90,Jiran92,Ye10a,Ye10b,Ye11a,Amram12,Rabkin14,Herz216,Naffouti16,Naffouti17,Kovalenko17} and theoretical studies \cite{Wong00,Dornel06,hyun2013quantitative,Jiang12,Jiang16,Kim15,Kan05,Srolovitz86a,Srolovitz86,Wang15,baojcm2022,Bao17,Bao17b,Zucker16}.
Various SSD models have been developed for cases involving isotropic surface energy \cite{Srolovitz86,Wong00,Dornel06,Jiang12,Zhao20}. However, the kinetic evolution during SSD is significantly influenced by crystalline anisotropy, as demonstrated by the experiments presented in \cite{Thompson12,Leroy16}.
%%%%%%%
Gaining a more comprehensive understanding of how crystalline anisotropy affects SSD is essential, as it not only leads to remarkable behaviors but also plays a significant role in utilizing dewetting to create intermediate structures for device fabrication. Accurately modeling SSD in materials with strong crystalline anisotropy remains a challenging problem in materials science, with significant implications for the manufacturing and reliability of nanoscale devices.
%%%%%%
In recent years, a variety of approaches have been explored to theoretically investigate the effects of surface energy anisotropy on SSD, as detailed in \cite{Dornel06,Pierre09b,Dufay11,Klinger11shape,Zucker13,Bao17,Jiang16,Wang15,Jiang19a,Zhao19b} and related references. 

Modeling and simulating faceting effects on surfaces is an increasingly complex task in nanotechnology, driven by non-convex and highly anisotropic surface energies, leading to ill-posed surface evolution equations. To address the issue of ill-posedness, a common approach is to regularize the energy with a curvature-dependent term \cite{kohn2007energy,torabi2009new,bao2017stable}, which also aligns with the underlying physical principles. However, this method leads to higher-order partial differential equations for surface variables, presenting considerable difficulties for numerical solutions, especially when developing algorithms that preserve the structure of the surface. A widely employed strategy to handle unstable orientations in such systems is the inclusion of a Willmore regularization term \cite{willmore1993}. Initially
proposed in \cite{angenent1989multiphase}, this regularization technique has been extensively used in many studies of strongly anisotropic
systems \cite{burger2007level,torabi2009new, fonseca2012motion,bao2017stable,chen2018efficient,maxwell2025level}, proving effective in stabilizing the surface evolution dynamics. Bao
et al. \cite{bao2017stable} demonstrated that in the strongly anisotropic setting, the evolution exhibits multiple stable equilibria. The
regularization method introduced in \cite{bao2017stable} may serve as an effective solver for the dynamical problem, which would help
explore the basins of attraction. We in this paper aim to further investigate the energy-stable algorithm for the regularized model of the strongly anisotropic SSD with axisymmetric geometry.

The interface surface that divides the vapor and the thin film is depicted as an open surface $\mS$, bounded by two closed curves, $\Gamma_i$ and $\Gamma_o$, on the substrate.
The original interfacial energy for the three-dimensional SSD can be defined by 
\begin{align}\label{eqn:ener1}
W_0(\mS) = \iint_\mS \gamma_{FV}(\vec N)d\mS - \underbrace{(\gamma_{VS} - \gamma_{FS})A(\Gamma_o/\Gamma_i)}_{\text{Substrate energy}},
\end{align}
where $\gamma_{FV}(\vec N)$ is the surface energy density of the thin film with $\vec N$ representing the unit out normal vector of surface, the constants $\gamma_{VS}$ and $\gamma_{FS}$ denote the surface energy densities of film/substrate and vapor/substrate, and $A(\Gamma_o/\Gamma_i)$ represents the area enclosed by the inner and outer contact lines. 

% Why add Willmore energy? 
When in the case of strong anisotropy, the interface governing equation induced by the surface energy will be ill-posed.
To make the interface governing equation well-posed, an efficient method is to add the Willmore energy to the original energy $W_0(\mS)$, given by 
\begin{align}\label{eqn:ener2}
W(\mS) = \iint_\mS \gamma_{FV}(\vec N)d\mS + \underbrace{\frac{\varepsilon ^2}{2}\iint_\mS\kappa_\mS^2d\mS}_{\text{Willmore\, energy}} - \underbrace{(\gamma_{VS} - \gamma_{FS})A(\Gamma_o/\Gamma_i)}_{\text{Substrate energy}},
\end{align}
where $\kappa_\mS$ denotes the mean curvature.

\begin{figure}[!htp]
\centering
\includegraphics[width=0.43\textwidth]{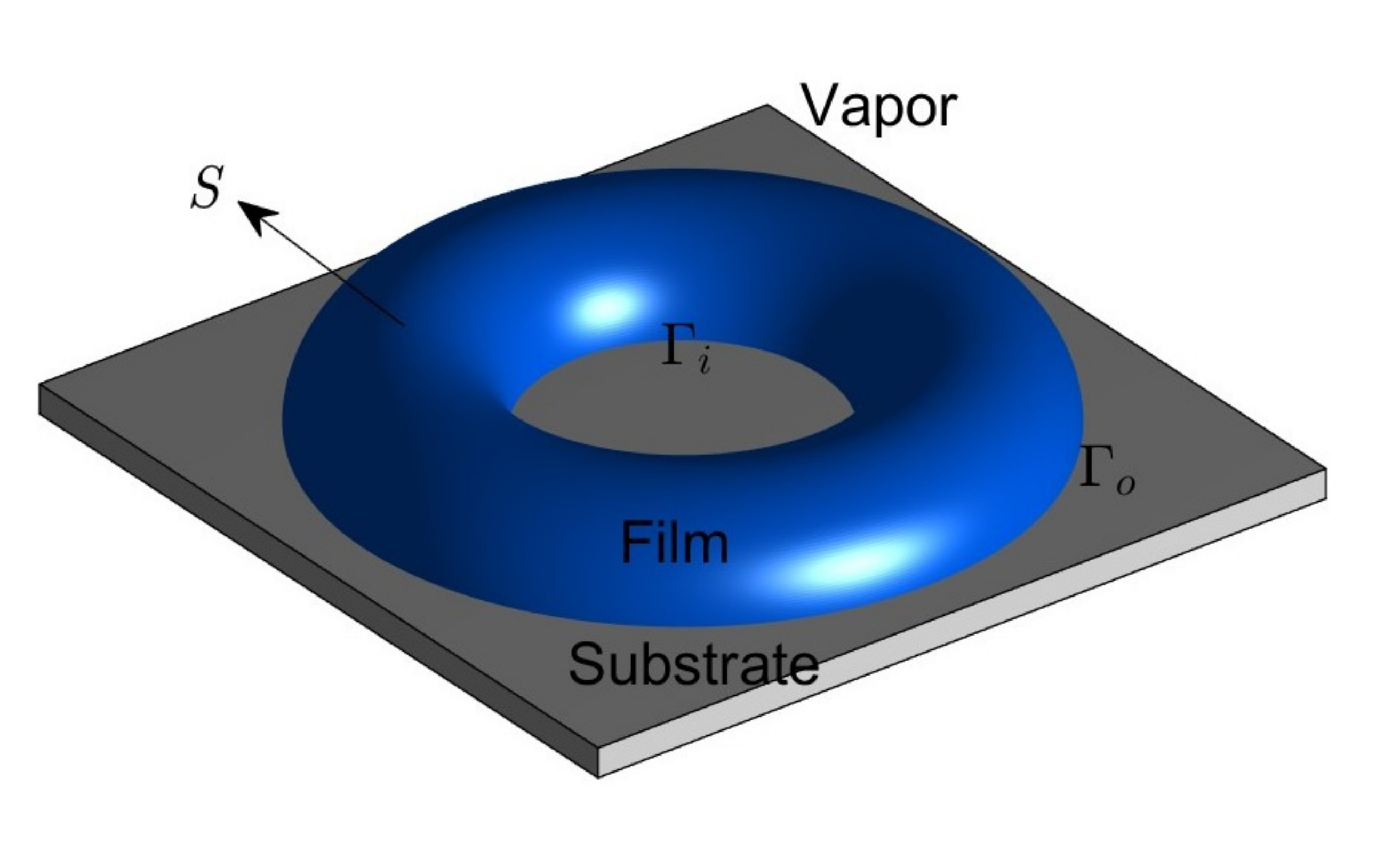}%\hspace{0.1cm}
\includegraphics[width=0.45\textwidth]{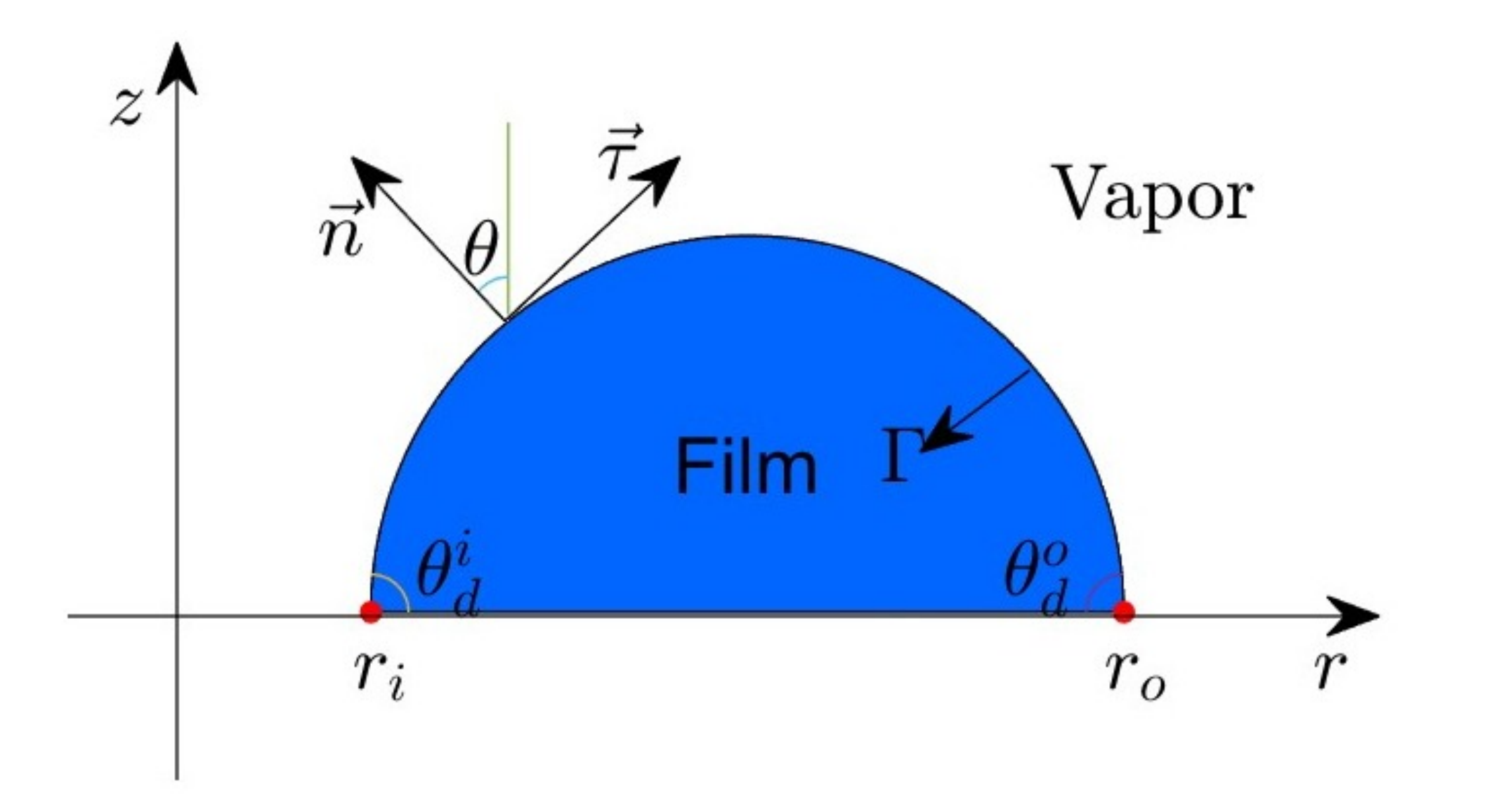}
\caption{A schematic illustration of the solid-state dewetting (left panel) a toroidal thin film on a flat substrate ; The cross-section of an axis-symmetric thin film in the cylindrical coordinate system $(r, z)$, $r_i$ and $r_o$ represent the radius of inner contact line and outer contact line respectively. }
\label{fig:shili}
\end{figure}

In this work, we assume the thin film is axisymmetric during the evolution. As shown in Figure \ref{fig:shili}, an axisymmetric film is situated on a flat substrate.
In this case, we can directly consider the evolution in the radial direction of the film, and the surface $\mS$ can be parameterized as
\begin{equation}
(s, \varphi)\to(\mS, \varphi) := (r(s)\cos\varphi, r(s)\sin\varphi, z(s)), 
\end{equation}
where $r(s)$ denotes the radial distance, $\varphi$ represents the azimuth angle, $z(s)$ is the film height in radial direction, and $s\in[0, L]$ is the arc length of radial direction curve. 
The axial symmetry can reduce this dependence to the orientation of curve in the radial direction.
We simply define the surface energy density of the thin film as \(\gamma(\theta) = \gamma_{FV}(\vec{N})\), where \(\theta\) is given by
\[
\theta = \arctan\left(\frac{z_s}{r_s}\right), \quad \gamma(\theta) = \gamma(-\theta), \quad \forall \theta \in [0, \pi], \quad \gamma(\theta) \in C^2([0, \pi]).
\]
Additionally, let \(r_o\) and \(r_i\) denote the lengths of the outer and inner contact lines, respectively.
Then the total interfacial energy \eqref{eqn:ener2} can be simplified to 
\begin{equation}\label{eqn:ener3}
W(\mS) = \iint_\mS \gamma(\theta)d\mS + \frac{\varepsilon ^2}{2}\iint_\mS\mu_\mS^2d\mS - (\gamma_{VS} - \gamma_{FS})( \pi r_o^2 - \pi r_i^2  ), 
\end{equation}
where 
$\mu_\mS = \kappa - \frac{\vec n\cdot\vec e_1}{r}$ represents the mean curvature in the case of axial symmetry, with $\kappa$ denoting the curvature,  $\vec n$ the outward unit normal vector in radial direction curve and $\vec e_1$ the unit vector along r-coordinate. If we denote $\Gamma := \vec X(s) = (r(s), z(s))$ as the open curve in radial direction, there holds
$\kappa = \vec X_{ss}\cdot\vec n$ with $\vec n = -\vec X_s^\bot $. For convenience, we introduce a time-independent variable $\rho\in\mathbb I = [0, 1]$ to parameterize the generating curve $\vec X(s)$ :
\begin{align}
\Gamma := \vec X(\rho) = (r(\rho), z(\rho)) : \mathbb I \to \mathbb R^2. 
\end{align}
According to the parametrization, the arc length s can be expressed as $s(\rho, t) = \int_0^\rho| \vec X_q| dq$. Furthermore, by differentiating both sides with respect to the parameter $\rho$, we have $s_\rho = | \vec X_\rho|$ and $ds = s_\rho d\rho = | \vec X_q | d\rho$.

The primary objective of this paper is to derive and numerically investigate a regularized sharp-interface model for strongly anisotropic SSD, assuming axisymmetric shapes, through thermodynamic variational principles for a new-defined interfacial energy. 
In \cite{Jiang19a}, 
using a Cahn–Hoffman $\xi$-vector formulation, based on the thermodynamic variation and smooth vector-field perturbation method, a sharp-interface model with weakly anisotropic surface energies was derived. Then, a PFEM was proposed for the sharp-interface model. However, the numerical method cannot be proved to be area-conservative and energy-stable. 
Based on thermodynamic variation principles, Zhao \cite{Zhao19} derived a sharp-interface model for the weakly anisotropic SSD of thin films on a flat substrate, assuming that the film morphology is axisymmetric. Similarly, the associated PFEM still lacks proof of its structure-preserving properties. 
In \cite{Zhao19b,Jiang19c}, two types of PFEMs were proposed for solving the
 morphological evolution of SSD of thin films on a at rigid substrate in three dimensions. 
In the aforementioned references, the governing equations for interface evolution were derived and numerically implemented; however, the numerical schemes do not exhibit structure-preserving properties.
Subsequently, Li and Bao \cite{li2021energy} proposed an energy-stable PFEM for surface diffusion flow and SSD with weakly anisotropic surface energy. Later, Li et al. \cite{li2023symmetrized} introduced an area-conservative and energy-stable PFEM for both weakly and strongly anisotropic SSD.
In \cite{li2024structure2}, we developed several structure-preserving algorithms for axisymmetric SSD with weakly and strongly anisotropic surface energies. 

When the surface stiffness \( H_{\gamma}(\vec{n})\tau \cdot \tau = \gamma(\theta) + \gamma''(\theta) < 0 \) for certain orientations \(\theta\), sharp corners may form in the equilibrium shape, corresponding to the strongly anisotropic case. In this case, the corresponding sharp-interface control equation becomes ill-posed. To resolve this issue, by adding Willmore regularization terms, the authors \cite{Jiang19a} constructed a regularized sharp-interface model for simulating SSD in two dimensions. A parametric finite element approximation was then developed based on the regularized system. This regularized method has been extensively developed in the literature (see \cite{burger2007level, torabi2009new, fonseca2012motion, bao2017stable, chen2018efficient, maxwell2025level} and the references therein).
%%%%%%%%
Despite its advantages, the Willmore-regularized sharp-interface model is highly complex, making it challenging to construct energy-stable schemes for the system.
In \cite{li2024energy}, 
by introducing two geometric
 relations inspired by \cite{bao2024energy}, 
we innovatively constructed an energy-stable PFEM for the Willmore-regularized sharp-interface model. However, there is currently no published work focusing on the regularized system for three-dimensional strongly anisotropic SSD, including both the model estabishment and the numerical simulation.
In this work, considering a special axisymmetric case, we regularize the total interfacial energy by introducing the well-known Willmore energy, which leads to a new set of interface governing equations for the strongly anisotropic SSD. 
This regularization ensures that the resulting sharp-interface model is well-posed. We introduce two surface energy matrices with $\theta$ as the variable. By combining two geometric relations in the axisymmetric version, we successfully construct a structure-preserving parametric finite element approximation for the Willmore-regularized system. Finally, we present several numerical simulations to demonstrate the accuracy, efficiency, and structure-preserving properties of the proposed numerical methods.

The rest of paper is organized as follows. In Section 2, we derive a sharp-interface model of SSD with strong anisotropies based on thermodynamic variations. In Section 3, We can obtain an equivalent sharp interface model using two geometric relations. Next, in Section 4, we present a weak formulation for this model and prove that the energy dissipation and volume conservation hold under this weak formulation. In Section 5, we construct a structure-preserving PFEM, and prove its energy stability and volume conservation. Subsequently, a large number of numerical experiments are presented in Section 6. Finally, we draw some conclusions in Section 7.

\section{The sharp-interface model}\label{sec2}
Given the parametrization of the surface $\mS(s, \varphi)$, we can directly calculate the two tangential vectors as follows
\begin{align}\label{eqn:tangential}
\mathcal T_1 = (r_s\cos \varphi, r_s\sin \varphi, z_s), \qquad \mathcal T_2 = (-r\sin \varphi, r\cos \varphi, 0). 
\end{align}
Then we can obtain the unit outer normal vector of the surface 
\begin{equation}
\vec N = \frac{\mathcal T_1 \times \mathcal T_2}{\left | \mathcal T_1 \times \mathcal T_2 \right | } = (-z_{s}\cos \varphi, -z_{s}\sin \varphi, r_{s}). 
\end{equation}
We consider $s = x_{1}$ and $\varphi = x_{2}$ as two parameters of the surface. Then the first fundamental form is given by 
\begin{equation}
I = Eds^2 + 2Fdsd\varphi + Gd\varphi ^2, 
\end{equation}
with $E = r_{s}^2 + z_{s}^2 = 1$, $F = 0$, $G = r^2$. Or it can be written in the metric tensor notation as 
\begin{equation}
(g_{ij}) = \begin{pmatrix}
 1 & 0\\
 0 & r^2
\end{pmatrix}, \quad g = r^2. 
\end{equation}

Let \(\mS^{\,\varepsilon}\) represent a small axisymmetric perturbation of the surface \(\mS\), defined by
\[
\mS^{\,\varepsilon} := \left(r^{\,\varepsilon}(s) \cos \varphi, r^{\,\varepsilon}(s) \sin \varphi, z^{\,\varepsilon}(s)\right).
\]
Since this surface is axially symmetric, the perturbed surface is generated by rotating the perturbed generating curve of the original surface. This perturbation in $(r,z)$ coordinates can be written as the following vector form: 
\begin{equation}
\Gamma^{\,\varepsilon  } := \vec X^{\,\varepsilon  } = (r^{\,\varepsilon  }(s), z^{\,\varepsilon  }(s)) = \vec X + \varepsilon  (r^{(1)}, z^{(1)}) = \vec X + \varepsilon  \vec X^{(1)},  
\end{equation}
where $\varepsilon$ denotes a perturbation parameter and $(r^{(1)}, z^{(1)})\in (\text{Lip}[0, L])^2$ represent the perturbations in the radial and axial directions, respectively. 
Assume that $g^{\,\varepsilon  }$ is the metric tensor of the perturbed surface $\mS^{\,\varepsilon}$, then it holds $g^{\,\varepsilon  } = [(r_s^{\,\varepsilon  })^2 + (z_s^{\,\varepsilon  })^2](r^{\,\varepsilon  })^2$. 
%Note that if we continue to use $s$ and $\varphi$ as the two parameters of the surface $\mS^{\,\varepsilon  }$, the corresponding metric tensor can be expressed as $g^{\,\varepsilon  } = [(r_s^{\,\varepsilon  })^2 + (z_s^{\,\varepsilon  })^2](r^{\,\varepsilon  })^2$. 
The total interfacial energy after perturbation can be represented as
\begin{align}
W^\varepsilon   
&= 
\iint_{\mS^\varepsilon  } \gamma(\theta^{\,\varepsilon  })d\mS^{\,\varepsilon  } + \underbrace{\frac{\varepsilon ^2}{2}\iint_{\mS^{\,\varepsilon  }} (\mu_\mS^{\,\varepsilon  })^2 d\mS^{\,\varepsilon  }}_{\text{Willmore\, energy}}   - \underbrace{(\gamma_{VS}-\gamma_{FS})\left[\pi (r^{\,\varepsilon  }_o)^2 - \pi (r^{\,\varepsilon  }_i)^2\right]}_{\text{Substrate\, energy}}  \nonumber \\
&= \int_{0}^{2\pi}\int_{0}^{L}\gamma(\theta ^{\,\varepsilon  })\sqrt{g^{\,\varepsilon  }}dsd\varphi + \frac{\varepsilon ^2}{2}\int_{0}^{2\pi}\int_{0}^{L}(\mu_\mS^{\,\varepsilon  })^2\sqrt{g^{\,\varepsilon  }}dsd\varphi - (\gamma_{VS}-\gamma_{FS})\left[\pi r^{\,\varepsilon  }(L)^2 - \pi r^{\,\varepsilon  }(0)^2\right] \nonumber \\
&= 
2\pi\int_{0}^{L}\gamma(\theta ^{\,\varepsilon  })\left | \vec X_s^{\,\varepsilon  } \right | r^{\,\varepsilon  } ds + \varepsilon ^2\pi\int_{0}^{L}(\mu_\mS^{\,\varepsilon  })^2\left | \vec X_s^{\,\varepsilon  } \right | r^{\,\varepsilon  } ds - (\gamma_{VS}-\gamma_{FS})\left[\pi r^{\,\varepsilon  }(L)^2 - \pi r^{\,\varepsilon  }(0)^2\right].
\end{align}

The unit tangential vector and outer normal vector of the curve are given as $\vec\tau = \vec X_s$ and $\vec n = -\vec\tau^\bot = -\vec X_s^\bot$, where $(\cdot)^\bot$ represents a 90-degree clockwise rotation of a vector. For later use, we present the Taylor expansions of the following terms at $\varepsilon = 0$: 
\begin{subequations}
\begin{align}
&r^{\,\varepsilon  } = r + r^{(1)}\varepsilon   + O(\varepsilon   ^2), \label{Taile1}\\
&\gamma(\theta ^{\,\varepsilon}) = \gamma(\theta) + \gamma'(\theta)(\vec X_s^{(1)}\cdot\vec n)\,\varepsilon + O(\varepsilon ^2), \label{Taile2}\\
&\left | \vec X_s^{\,\varepsilon} \right | = 1 + (\vec X_s^{(1)}\cdot\vec\tau)\varepsilon + O(\varepsilon ^2), \label{Taile3}\\
&\kappa^{\,\varepsilon} = \kappa + \left[\theta_s^{(1)} - \theta_s(\vec X_s^{(1)}\cdot\vec\tau)\right]\,\varepsilon + O(\varepsilon ^2), \label{Taile4}\\
&\vec n^{\,\varepsilon} = \vec n + \left[-(\vec X_s^{(1)}\cdot\vec\tau)\vec n - (\vec X_s^{(1)})^\bot\right]\,\varepsilon + O(\varepsilon ^2). \label{Taile5}
\end{align}
\end{subequations}
By using \eqref{Taile1}-\eqref{Taile5}, we can calculate the first variation of the total interfacial energy, given by 
\begin{align}\label{W^(1)}
W^{(1)} 
&= 2\pi\int_{0}^{L}\left[\gamma'(\theta)(\vec X_s^{(1)}\cdot\vec n)r + \gamma(\theta)(\vec X_s^{(1)}\cdot\tau)r + \gamma(\theta)r^{(1)}\right] ds \nonumber \\
&~~~~+ \varepsilon ^2\pi\int_{0}^{L}\left[r^{(1)}\mu_\mS^2 + r\mu_\mS^2(\vec X_s^{(1)}\cdot\vec\tau)\right] ds \nonumber \\
&~~~~+ 2\varepsilon ^2\pi\int_{0}^{L} 
\left[r\theta_s^{(1)}\mu_\mS - r\theta_s(\vec X_s^{(1)}\cdot\vec\tau)\mu_\mS\right] ds \nonumber \\
&~~~~+ 2\varepsilon ^2\pi\int_{0}^{L}
\left[\vec n\cdot\vec e_1(\vec X_s^{(1)}\cdot\vec\tau)\mu_\mS + (\vec X_s^{(1)})^\bot\cdot\vec e_1\mu_\mS + \frac{\vec n\cdot\vec e_1 r^{(1)}}{r}\mu_\mS\right] ds \nonumber \\
&~~~~- 2\pi(\gamma_{VS} - \gamma_{FS})\left[r(L)r^{(1)}(L) - r(0)r^{(1)}(0)\right]  \nonumber \\
&=: \mathbb{I}_1 + \mathbb{I}_2 + \mathbb{I}_3 + \mathbb{I}_4 + \mathbb{I}_5. 
\end{align}
Using the relationships below 
\begin{align*}
&\kappa = \vec X_{ss}\cdot\vec n = \theta_s, \quad \vec n_s = -\kappa\,\vec \tau, \quad \vec \tau_s = \kappa\,\vec n, \quad r^{(1)} = r_s\,\vec X^{(1)}\cdot\vec\tau - z_s\,\vec X^{(1)}\cdot\vec n, \\
&\vec e_1 = (\vec\tau\cdot\vec e_1)\,\vec\tau + (\vec n\cdot\vec e_1)\,\vec n, \quad \vec n^\bot = \vec\tau, \quad \vec\tau^\bot = -\vec n,
\end{align*}
and by applying integration by parts in \eqref{W^(1)}, we can obtain
\begin{subequations}
\begin{align}
\mathbb{I}_1 &= 2\pi\int_{0}^{L}\left[\gamma'(\theta)(\vec X_s^{(1)}\cdot\vec n)r + \gamma(\theta)(\vec X_s^{(1)}\cdot\vec\tau)r + \gamma(\theta)r^{(1)}\right] ds \nonumber \\
&= 2\pi\int_{0}^{L}\left[ -(\gamma(\theta)\vec\tau + \gamma'(\theta)\vec n)_s\cdot\vec X^{(1)}r - (\gamma(\theta)\vec\tau + \gamma'(\theta)\vec n)\cdot\vec X^{(1)}r_s + \gamma(\theta)r^{(1)} \right] ds \nonumber \\
&~~~~ +2\pi\left[ (\gamma(\theta)\vec \tau + \gamma'(\theta)\vec n)\cdot\vec X^{(1)}r \right]\bigg|_{s=0}^{s=L} \nonumber \\
&= 2\pi\int_{0}^{L}\left[ -(\gamma(\theta) + \gamma''(\theta))\kappa r(\vec X^{(1)}\cdot\vec n) - (\gamma(\theta)z_s + \gamma'(\theta)r_s)(\vec X^{(1)}\cdot\vec n) \right]ds \nonumber \\
&~~~~ +2\pi\left[ (\gamma(\theta)\vec \tau + \gamma'(\theta)\vec n)\cdot\vec X^{(1)}r \right]\bigg|_{s=0}^{s=L}, \\
\mathbb{I}_2 &= \varepsilon ^2\pi\int_{0}^{L}\left[r^{(1)}\mu_\mS^2 + r\mu_\mS^2(\vec X_s^{(1)}\cdot\vec\tau)\right] ds \nonumber \\
&= \varepsilon ^2\pi\int_{0}^{L}\left[ r_s\,\mu_\mS^2(\vec X^{(1)}\cdot\vec\tau) - z_s\,\mu_\mS^2(\vec X^{(1)}\cdot\vec n) - (r\,\mu_\mS^2\,\vec\tau)_s\cdot\vec X^{(1)}\right] ds + \varepsilon ^2\pi\left[ r\,\mu_\mS^2\,(\vec X^{(1)}\cdot\vec\tau) \right]\bigg|_{s=0}^{s=L} \nonumber \\
&= \varepsilon ^2\pi\int_{0}^{L}\left[ -(z_s\,\mu_\mS^2 - r\,\kappa\,\mu_\mS^2)(\vec X^{(1)}\cdot\vec n) -2r\,\mu_\mS\,(\mu_\mS)_s(\vec X^{(1)}\cdot\vec\tau) \right] ds + \varepsilon ^2\pi\left[ r\,\mu_\mS^2\,(\vec X^{(1)}\cdot\vec\tau) \right]\bigg|_{s=0}^{s=L}, \\
\mathbb{I}_3 &= 2\varepsilon ^2\pi\int_{0}^{L} 
\left[r\,\theta_s^{(1)}\mu_\mS - r\,\theta_s(\vec X_s^{(1)}\cdot\vec\tau)\,\mu_\mS\right] ds \nonumber \\
&= \varepsilon ^2\pi\int_{0}^{L}\left[ (2r\,\kappa\,\mu_\mS\,\vec\tau)_s\cdot\vec X^{(1)} -(2r\,\mu_\mS)_s\,\theta^{(1)} \right] ds + \left[ 2r\,\mu_\mS\,\theta^{(1)} - 2r\,\kappa\,\mu_\mS(\vec X^{(1)}\cdot\vec\tau) \right]\bigg|_{s=0}^{s=L} \nonumber \\
&= \varepsilon ^2\pi\int_{0}^{L}\left[ (2r\,\kappa\,\mu_\mS\,\vec\tau)_s\cdot\vec X^{(1)} + (2r_s\,\mu_\mS\,\vec n)_s\cdot\vec X^{(1)} + (2r\,(\mu_\mS)_s\,\vec n)_s\cdot\vec X^{(1)} \right] ds \nonumber \\
&~~~~ + \varepsilon ^2\pi\left[ 2r\,\mu_\mS\,\theta^{(1)} - 2r\,\kappa\,\mu_\mS(\vec X^{(1)}\cdot\vec\tau) - 2r_s\,\mu_\mS\,(\vec X^{(1)}\cdot\vec n) - (2r\,(\mu_\mS)_s\,(\vec X^{(1)}\cdot\vec n))
\right]\bigg|_{s=0}^{s=L} \nonumber \\
&= \varepsilon ^2\pi\int_{0}^{L}\left[ (2r_{ss}\mu_\mS + 4r_s(\mu_\mS)_s + 2r\,(\mu_\mS)_{ss} + 2r\,\kappa^2\,\mu_\mS )(\vec X\cdot\vec n) +2r\,\kappa_s\,\mu_\mS (\vec X^{(1)}\cdot\vec \tau) \right] ds \nonumber \\
&~~~~ + \varepsilon ^2\pi\left[ 2r\,\mu_\mS\,\theta^{(1)} - 2r\,\kappa\,\mu_\mS(\vec X^{(1)}\cdot\vec\tau) - 2r_s\,\mu_\mS\,(\vec X^{(1)}\cdot\vec n) - (2r\,(\mu_\mS)_s\,(\vec X^{(1)}\cdot\vec n))
\right]\bigg|_{s=0}^{s=L},  \\
\mathbb{I}_4 &= 2\varepsilon ^2\pi\int_{0}^{L}
\left[\vec n\cdot\vec e_1(\vec X_s^{(1)}\cdot\vec\tau)\mu_\mS + (\vec X_s^{(1)})^\bot\cdot\vec e_1\mu_\mS + \frac{\vec n\cdot\vec e_1 r^{(1)}}{r}\mu_\mS\right] ds \nonumber \\
&= \varepsilon ^2\pi\int_{0}^{L}\left[ -(2\vec n\cdot\vec e_1\,\vec\tau\,\mu_\mS)_s\cdot\vec X^{(1)} -2(\vec X^{(1)})^\bot\cdot\vec e_1\,(\mu_\mS)_s + 2\frac{\vec n\cdot\vec e_1\,r_s\,\mu_\mS}{r}(\vec X^{(1)}\cdot\vec\tau) - 2\frac{\vec n\cdot\vec e_1\,z_s\,\mu_\mS}{r}(\vec X^{(1)}\cdot\vec n) \right] ds \nonumber \\
&~~~~ + \varepsilon ^2\pi\left[2\vec n\cdot\vec e_1\, 
\mu_\mS\,(\vec X^{(1)}\cdot\vec\tau) + 2(\vec X^{(1)})^\bot\cdot\vec e_1\,\mu_\mS\right]\bigg|_{s=0}^{s=L} \nonumber \\
&= \varepsilon ^2\pi\int_{0}^{L}\left[ -(2\vec n\cdot\vec e_1\,\vec\tau\,\mu_\mS)_s\cdot\vec X^{(1)} -2(\vec X^{(1)})^\bot\cdot((\vec\tau\cdot\vec e_1)\,\vec\tau + (\vec n\cdot\vec e_1)\,\vec n)\,(\mu_\mS)_s\right] \nonumber \\
&~~~~+ \varepsilon ^2\pi\int_{0}^{L}\left[2\frac{\vec n\cdot\vec e_1\,r_s\,\mu_\mS}{r}(\vec X^{(1)}\cdot\vec\tau) - 2\frac{\vec n\cdot\vec e_1\,z_s\,\mu_\mS}{r}(\vec X^{(1)}\cdot\vec n) \right] ds  + \varepsilon ^2\pi\left[2\vec n\cdot\vec e_1\, 
\mu_\mS\,(\vec X^{(1)}\cdot\vec\tau) + 2(\vec X^{(1)})^\bot\cdot\vec e_1\,\mu_\mS\right]\bigg|_{s=0}^{s=L} \nonumber \\
&= \varepsilon ^2\pi\int_{0}^{L}\left[ (2z_s\,\kappa\,\mu_\mS + 2\frac{z_s^2\,\mu_\mS}{r} -2r_s\,(\mu_\mS)_s )(\vec X^{(1)}\cdot\vec n) + (2r_s\,\kappa\,\mu_\mS -2\frac{z_s\,r_s\,\mu_\mS}{r})(\vec X\cdot\vec\tau) \right] ds \nonumber \\
&~~~~ + \varepsilon ^2\pi\left[2\vec n\cdot\vec e_1\, 
\mu_\mS\,(\vec X^{(1)}\cdot\vec\tau) + 2(\vec X^{(1)})^\bot\cdot\vec e_1\,\mu_\mS \right]\bigg|_{s=0}^{s=L}, \nonumber \\
\mathbb I_5 &= - 2\pi(\gamma_{VS} - \gamma_{FS})\left[r(L)r^{(1)}(L) - r(0)r^{(1)}(0)\right]. 
\end{align}
\end{subequations}
Summing the five terms above, we have
\begin{align}
W^{(1)} &= \mathbb{I}_1 + \mathbb{I}_2 + \mathbb{I}_3 + \mathbb{I}_4 +\mathbb{I}_5 \nonumber \\
&= 2\pi\int_{0}^{L}\left[ -(\gamma(\theta) + \gamma''(\theta))\kappa r(\vec X^{(1)}\cdot\vec n) - (\gamma(\theta)z_s + \gamma'(\theta)r_s)(\vec X^{(1)}\cdot\vec n) \right]ds \nonumber \\
&~~~~ + \pi\int_{0}^{L}\left[ \varepsilon ^2(2r\,\kappa_s\,\mu_\mS + 2r_s\,\kappa\,\mu_\mS - 2\frac{z_s\,r_s\,\mu_\mS}{r} - 2r\,\mu_\mS\,(\mu_\mS)_s)(\vec X^{(1)}\cdot\vec\tau) \right] ds \nonumber \\ 
&~~~~ + \pi\int_{0}^{L}\left[ \varepsilon ^2(-z_s\,\mu_\mS^2 + 2r_{ss}\,\mu_\mS + 2r_s\,(\mu_\mS)_s + 2r\,(\mu_\mS)_{ss} + 2r\kappa^2\,\mu_\mS + 2z_s\,\kappa\,\mu_\mS + 2\frac{z_s^2\,\mu_\mS}{r} - r\,\kappa\,\mu_\mS^2)(\vec X^{(1)}\cdot\vec n) \right] ds \nonumber \\ 
&~~~~ + \left[ 2\pi(\gamma(\theta)\vec \tau + \gamma'(\theta)\vec n)\cdot\vec X^{(1)}r + \varepsilon ^2\pi\left(r\,\mu_\mS^2\,(\vec X^{(1)}\cdot\vec\tau) + 2\vec n\cdot\vec e_1\, 
\mu_\mS\,(\vec X^{(1)}\cdot\vec\tau) + 2(\vec X^{(1)})^\bot\cdot\vec e_1\,\mu_\mS\right)\right]\bigg|_{s=0}^{s=L} \nonumber \\
&~~~~ + \varepsilon ^2\pi\left[ 2r\,\mu_\mS\,\theta^{(1)} - 2r\,\kappa\,\mu_\mS(\vec X^{(1)}\cdot\vec\tau) - 2r_s\,\mu_\mS\,(\vec X^{(1)}\cdot\vec n) - (2r\,(\mu_\mS)_s\,(\vec X^{(1)}\cdot\vec n))
\right]\bigg|_{s=0}^{s=L}  \nonumber  \\
&~~~~ - 2\pi(\gamma_{VS} - \gamma_{FS})\left[r(L)r^{(1)}(L) - r(0)r^{(1)}(0)\right]. 
\end{align}
Noting the relationships 
\begin{align*}
    &r_{ss} = \vec X_{ss}\cdot\vec e_1 = (\vec X_{ss}\cdot\vec n)(\vec n\cdot\vec e_1) = \kappa\,(\vec n\cdot\vec e_1) = -z_s\,\kappa,\\
    &z_{ss} = (\vec X_{ss})^\bot\cdot\vec e_1 = ((\vec X_{ss})^\bot\cdot\vec \tau)(\vec \tau\cdot\vec e_1) = (\vec X_{ss}\cdot\vec n)(\vec \tau\cdot\vec e_1) = r_s\,\kappa,
\end{align*}
and denoting $\kappa_m:=\kappa(\mu_\mS - \kappa)$,  we can obtain the first variation of the total interfacial energy, given by 
\begin{align}
W^{(1)} &= 2\pi\int_{0}^{L}\left[ -(\gamma(\theta) + \gamma''(\theta))\kappa r(\vec X^{(1)}\cdot\vec n) - (\gamma(\theta)z_s + \gamma'(\theta)r_s)(\vec X^{(1)}\cdot\vec n) \right]ds \nonumber \\
&~~~~ + \pi\int_{0}^{L}\left[ \varepsilon ^2(r\,\mu_\mS^3 + 2r_s\,(\mu_\mS)_s + 2r\,(\mu_\mS)_{ss} + 4r\,\kappa^2\mu_\mS - 4r\,\kappa\,\mu_\mS^2 )(\vec X^{(1)}\cdot\vec n) \right]ds \nonumber \\ 
&~~~~ + \left[ 2\pi(\gamma(\theta)\vec \tau + \gamma'(\theta)\vec n)\cdot\vec X^{(1)}r + \varepsilon ^2\pi\left(r\,\mu_\mS^2\,(\vec X^{(1)}\cdot\vec\tau) + 2\vec n\cdot\vec e_1\, 
\mu_\mS\,(\vec X^{(1)}\cdot\vec\tau) + 2(\vec X^{(1)})^\bot\cdot\vec e_1\,\mu_\mS\right)\right]\bigg|_{s=0}^{s=L} \nonumber \\
&~~~~ + \varepsilon ^2\pi\left[ 2r\,\mu_\mS\,\theta^{(1)} - 2r\,\kappa\,\mu_\mS(\vec X^{(1)}\cdot\vec\tau) - 2r_s\,\mu_\mS\,(\vec X^{(1)}\cdot\vec n) - (2r\,(\mu_\mS)_s\,(\vec X^{(1)}\cdot\vec n))
\right]\bigg|_{s=0}^{s=L}  \nonumber  \\
&~~~~ - 2\pi(\gamma_{VS} - \gamma_{FS})\left[r(L)r^{(1)}(L) - r(0)r^{(1)}(0)\right] \nonumber \\
&= 2\pi\int_{0}^{L}\left[ -(\gamma(\theta) + \gamma''(\theta))\kappa r - (\gamma(\theta)z_s + \gamma'(\theta)r_s) + \varepsilon ^2((r\,(\mu_\mS)_s)_s - 2r\,\mu_\mS\,\kappa_m + \frac{1}{2}r\,\mu_\mS^3 ) \right](\vec X^{(1)}\cdot\vec n)ds \nonumber \\
&~~~~ + \left[ 2\pi(\gamma(\theta)\vec \tau + \gamma'(\theta)\vec n)\cdot\vec X^{(1)}r + \varepsilon ^2\pi\left(r\,\mu_\mS^2\,(\vec X^{(1)}\cdot\vec\tau) + 2\vec n\cdot\vec e_1\, 
\mu_\mS\,(\vec X^{(1)}\cdot\vec\tau) + 2(\vec X^{(1)})^\bot\cdot\vec e_1\,\mu_\mS\right)\right]\bigg|_{s=0}^{s=L} \nonumber \\
&~~~~ + \varepsilon ^2\pi\left[ 2r\,\mu_\mS\,\theta^{(1)} - 2r\,\kappa\,\mu_\mS(\vec X^{(1)}\cdot\vec\tau) - 2r_s\,\mu_\mS\,(\vec X^{(1)}\cdot\vec n) - (2r\,(\mu_\mS)_s\,(\vec X^{(1)}\cdot\vec n))
\right]\bigg|_{s=0}^{s=L}  \nonumber  \\
&~~~~ - 2\pi(\gamma_{VS} - \gamma_{FS})\left[r(L)r^{(1)}(L) - r(0)r^{(1)}(0)\right]. 
\end{align}
To ensure energy dissipation for any perturbation, 
we must require  $(2r\,\mu_\mS\,\theta^{(1)})\bigg|_{s=0}^{s=L}=0$. Therefore, we need to impose a zero curvature condition:
\begin{align}
    \mu_\mS(0, t)=0,\qquad \mu_\mS(L, t) = 0,\qquad t\geq 0.
\end{align}
Because the substrate is flat and the contact points move along the substrate. The perturbation velocity field at the contact points should satisfy $\vec X^{(1)}(s = 0)//\vec e_1$ and $\vec X^{(1)}(s = L)//\vec e_1$. By utilizing the zero curvature condition 
and the following two contact point conditions: 
\begin{align}
&\vec \tau \big|_{s = 0} = (\cos\theta_d^{\,i}, \sin\theta_d^{\,i}), \qquad \vec n \big|_{s = 0} = (-\sin\theta_d^{\,i}, \cos\theta_d^{\,i}), \qquad \vec X^{(1)}\big|_{s = 0} = (r^{(1)}(0), 0),  \nonumber \\
&\vec \tau \big|_{s = L} = (\cos\theta_d^{\,o}, \sin\theta_d^{\,o}), \qquad \vec n \big|_{s = L} = (-\sin\theta_d^{\,o}, \cos\theta_d^{\,o}), \qquad \vec X^{(1)}\big|_{s = L} = (r^{(1)}(L), 0),  \nonumber 
\end{align}
we can obtain 
\begin{align}
W^{(1)} &= 2\pi\int_{0}^{L}\left[ -(\gamma(\theta) + \gamma''(\theta))\kappa r - (\gamma(\theta)z_s + \gamma'(\theta)r_s) + \varepsilon ^2((r\,(\mu_\mS)_s)_s - 2r\,\mu_\mS\,\kappa_m + \frac{1}{2}r\,\mu_\mS^3 ) \right](\vec X^{(1)}\cdot\vec n)ds \nonumber \\
&~~~~ + 2\pi r(L)r^{(1)}(L)\left[ \gamma(\theta_d^{\,o})\cos\theta_d^{\,o} - \gamma'(\theta_d^{\,o})\sin\theta_d^{\,o} - (\gamma_{VS} - \gamma_{FS}) - \varepsilon ^2(\mu_\mS)_s n_1) \right] \nonumber \\
&~~~~ - 2\pi r(0)r^{(1)}(0)\left[ \gamma(\theta_d^{\,i})\cos\theta_d^{\,i} - \gamma'(\theta_d^{\,i})\sin\theta_d^{\,i} - (\gamma_{VS} - \gamma_{FS}) - \varepsilon ^2(\mu_\mS)_s n_1) \right] \nonumber \\ 
&= \iint_{\mS}\left[ -(\gamma(\theta) + \gamma''(\theta))\kappa - \frac{\gamma(\theta)z_s + \gamma'(\theta)r_s}{r} +\varepsilon ^2\frac{(r(\mu_\mS)_s)_s - 2r\,\mu_\mS\,\kappa_m + \frac{1}{2}r\,\mu_\mS^3 }{r} \right](\vec X^{(1)}\cdot\vec n)d\mS \nonumber \\
&~~~~ + \int_{\Gamma_o} r^{(1)}(L)\left[ \gamma(\theta_d^{\,o})\cos\theta_d^{\,o} - \gamma'(\theta_d^{\,o})\sin\theta_d^{\,o} - (\gamma_{VS} - \gamma_{FS} ) - \varepsilon ^2(\mu_\mS)_s n_1 \right]d\,\Gamma \nonumber \\
&~~~~ - \int_{\Gamma_i} r^{(1)}(0)\left[ \gamma(\theta_d^{\,i})\cos\theta_d^{\,i} - \gamma'(\theta_d^{\,i})\sin\theta_d^{\,i} - (\gamma_{VS} - \gamma_{FS}) - \varepsilon ^2(\mu_\mS)_s n_1\right]d\,\Gamma. 
\end{align}
Then we can directly derive the variation of total energy in relation to the surface and the two contact lines as follows:
\begin{subequations}\label{bianhua}
\begin{align} 
&\frac{\delta W}{\delta \mS} = -(\gamma(\theta) + \gamma''(\theta))\kappa - \frac{\gamma(\theta)\, z_s + \gamma'(\theta)\, r_s}{r} + \varepsilon ^2\frac{ (r\,(\mu_\mS)_s)_s -2r\,\mu_\mS\,\kappa_m + \frac{1}{2}r\,\mu_\mS^3}{r},  \\
&\frac{\delta W}{\delta \Gamma_o} = \gamma(\theta_d^{\,o})\cos\theta_d^{\,o} - \gamma'(\theta_d^{\,o})\sin\theta_d^{\,o} - (\gamma_{VS} - \gamma_{FS}) - \varepsilon ^2(\mu_\mS)_s n_1,  \\
&\frac{\delta W}{\delta \Gamma_i} = -\left[\gamma(\theta_d^{\,i})\cos\theta_d^{\,i} - \gamma'(\theta_d^{\,i})\sin\theta_d^{\,i} - (\gamma_{VS} - \gamma_{FS}) - \varepsilon ^2(\mu_\mS)_s n_1\right]. 
\end{align} 
\end{subequations}
\begin{rem}
   The variation \eqref{bianhua} is based on the thin film with holes, which has two contact lines. For the film without holes, it does not have internal contact lines, so the following equations must be satisfied at the boundary: 
   \begin{equation}
      \vec X^{(1)}\bigg|_{s = 0} = (0, z^{(1)}(0)), \qquad \vec X^{(1)}\bigg|_{s = L} = (r^{(1)}(L), 0). 
   \end{equation}
   When this film is hole-free, meaning that there are no internal contact lines, using integration by parts will not produce boundary term at $s = 0$.
\end{rem}
By recalling anisotropic Gibbs-Thomson relation \cite{Sutton95}, the chemical potential is defined as 
\begin{equation}
  \mu = \Omega_{\,0}\frac{\delta W}{\delta \mS} = \Omega_{\,0}\left[ -(\gamma(\theta) + \gamma''(\theta))\,\kappa - \frac{\gamma(\theta)z_s + \gamma'(\theta)r_s}{r} + \varepsilon ^2\frac{ (r\,(\mu_\mS)_s)_s -2r\,\mu_\mS\,\kappa_m + \frac{1}{2}r\,\mu_\mS^3 }{r} \right], 
\end{equation}
where $\Omega_{\,0}$ denotes the atomic volume of the thin film material. According to Fick's laws of diffusion, we can obtain the normal velocity given by surface diffusion \cite{Mullins57, Cahn74}
\begin{equation}
\vec j = -\frac{D_s\,v}{k_B\,T_e}\bigtriangledown_s\mu, \qquad v_n = -\Omega_0(\bigtriangledown_s\cdot\vec j) = \frac{D_s\,v\,\Omega_0}{k_B\,T_e}\bigtriangledown_s^2\mu,  
\end{equation}
where $\vec j$ represents mass flux, $D_s$ denotes surface diffusivity, $k_B\,T_e$ is thermal energy, $v$ is the number of diffusing atoms per unit area, and $\bigtriangledown_s$ represents surface gradient. The two contact lines $\Gamma_i$ and $\Gamma_o$, move along the substrate, with the velocities $v_c^i$ and $v_c^o$ representing energy gradient flow, as defined by the time-dependent Ginzburg–Landau kinetic equations:
\begin{subequations}
    \begin{align}
       &v_c^o = -\eta\frac{\delta W}{\delta \Gamma_o} = -\eta\left[ \gamma(\theta_d^{\,o})\cos\theta_d^{\,o} - \gamma'(\theta_d^{\,o})\sin\theta_d^{\,o} - (\gamma_{VS} - \gamma_{FS}) - \varepsilon ^2(\mu_\mS)_s n_1) \right],  \\
       &v_c^i = -\eta\frac{\delta W}{\delta \Gamma_i} = \eta\left[\gamma(\theta_d^{\,i})\cos\theta_d^{\,i} - \gamma'(\theta_d^{\,i})\sin\theta_d^{\,i} - (\gamma_{VS} - \gamma_{FS}) - \varepsilon ^2(\mu_\mS)_s n_1) \right], 
    \end{align}
\end{subequations}
where $\eta\in(0, +\infty)$ represents the contact line mobility. 
Then we take the characteristic length scale and characteristic surface energy scale as $L$ and $\gamma_0$ respectively, and choose the time scale as $\frac{L^4}{B\,\gamma_0}$ with $B = \frac{D_s\,v\,\Omega_0^2}{k_B\,T_e}$. In addition, we select the contact line mobility as $\frac{B}{L^3}$. Since this model is axisymmetric, we further have $$\bigtriangledown_s^2\mu = \frac{1}{\sqrt{g}}\partial_i(\sqrt{g}g^{ij}\partial_j\mu) = \frac{1}{r}(r\mu_s)_s.$$ 
Then the Willmore regularized sharp interface model for SSD, considering anisotropic surface energy in three dimensions with axial symmetry, can be expressed in the following dimensionless form: 
\begin{subequations}\label{governingeq}
\begin{align}
 &r\,\vec X_t\cdot\vec n = (r\,\mu_s)_s, \quad 0<s<L(t), \quad t>0,  \\
 &\mu = -(\gamma(\theta) + \gamma''(\theta))\kappa - \frac{\gamma(\theta)\, z_s + \gamma'(\theta)\, r_s}{r} + \varepsilon ^2\frac{ (r\,(\mu_\mS)_s)_s -2r\,\mu_\mS\,\kappa_m + \frac{1}{2}r\,\mu_\mS^3}{r},  \\
 &\kappa = \vec X_{ss}\cdot\vec n, \quad \mu_\mS = \kappa - \frac{\vec n\cdot\vec e_1}{r}, \quad \kappa_m = \kappa(\mu_\mS - \kappa), \quad \vec n = -\vec X_s^\bot, 
\end{align}
\end{subequations}
where $\Gamma(t) := \vec X(s, t) = (r(s, t), z(s, t))$ is the generating curve of surface $\mS$, $L := L(t)$ denotes total arc length of open curve $\Gamma(t)$, $\mu(s, t)$ is chemical potential, $\kappa(s, t)$ is curvature of curve, $\mu_\mS(s, t)$ represents mean curvature, $\kappa_m$ is Gaussian curvature of surface $\mS$, $\vec n = (n_1, n_2) = (-z_s, r_s)$ is outward unit normal vector, and $0< \varepsilon\ll 1$ is a small regularization parametrization. The initial data is given as
\begin{equation}\label{initial}
\vec X(s, 0) := \vec X_0(s) = (r(s, 0), z(s, 0)), \qquad 0\le s\le L_0:=L(0). 
\end{equation}
This governing equation \eqref{governingeq}  satisfies the following boundary conditions: 

(i) contact line condition 
\begin{equation}\label{eqn:boundry1}
z(L, t) = 0, \quad \left\{\begin{matrix}
 z(0, t) = 0, & ~~~~\text{if}~~r(0, t)>0,\\
 z_s(0, t) = 0, & \text{otherwise}, 
\end{matrix}\right. \quad t\ge 0; 
\end{equation}

(ii) relaxed contact angle condition 
\begin{equation}\label{eqn:boundry2}
r_t(L, t) = -\eta f^{\varepsilon ^2}(\theta_d^{\,o}; \sigma), \quad \left\{\begin{matrix}
 r_t(0, t) = \eta f^{\varepsilon ^2}(\theta_d^{\,i}; \sigma), & ~~~~\text{if}~~r(0, t)>0,\\
 r(0, t) = 0, & \text{otherwise}, 
\end{matrix}\right. \quad t\ge 0; 
\end{equation}

(iii) zero-mass flux condition 
\begin{equation}\label{eqn:boundry3}
\mu_s(0, t) = 0, \quad \mu_s(L ,t) = 0, \quad t\ge 0; 
\end{equation}

(iv) zero-curvature condition 
\begin{equation}\label{eqn:boundry4}
\mu_\mS(0, t) = 0, \quad \mu_\mS(L, t) = 0, \quad t\ge 0, 
\end{equation}
where the function $f^{\varepsilon ^2}(\theta; \sigma)$ is defined by
\begin{equation}
f^{\varepsilon ^2}(\theta; \sigma) = \gamma(\theta)\cos\theta - \gamma'(\theta)\sin\theta - \sigma - \varepsilon ^2(\mu_\mS)_s\sin\theta, \quad \theta\in [-\pi, \pi], \quad \sigma = \frac{\gamma_{VS} - \gamma_{FS}}{\gamma_0}, 
\end{equation}
satisfying $\lim\limits_{\varepsilon ^2\to 0^+}f^{\varepsilon ^2}(\theta; \sigma)=f(\theta; \sigma)$.

We define $\text{vol}(\vec X(t))$ as the volume/mass of the thin film on the substrate, and $W(t)$ 
 as the total energy. For the system \eqref{governingeq} together with the boundary conditions (i)-(iv), the results can be obtained by calculating surface integrals
\begin{equation}
\text{vol}(\vec X(t)) = 2\pi\int_0^{L(t)}rzr_sds, \quad W(t) = 2\pi\int_0^{L(t)}r\gamma(\theta)ds + \varepsilon ^2\pi\int_0^{L(t)}r\mu_\mS^2ds - \sigma\pi(r_o^2 - r_i^2). 
\end{equation}
By directly differentiating with respect to the time variable t, we have 
\begin{equation}
\text{vol}(\vec X(t)) = \text{vol}(\vec X(0)), \qquad W(t_2)\le W(t_1)\le W(0), \qquad t_2\ge t_1\ge 0,
\end{equation}
i.e. volume conservation and energy dissipation laws. 

\section{A new geometric system}\label{sec3}

We define a surface energy matrix as 
\begin{equation}
\mat B_q(\theta) = \begin{pmatrix}
 \gamma(\theta) & -\gamma'(\theta) \\
 \gamma'(\theta) & \gamma(\theta) 
\end{pmatrix}\begin{pmatrix}
 \cos2\theta & \sin2\theta \\
 \sin2\theta & -\cos2\theta 
\end{pmatrix}^{1-q} + \mathscr S(\theta)\left[ \frac{1}{2}\mat I - \frac{1}{2}\begin{pmatrix}
 \cos2\theta & \sin2\theta \\
 \sin2\theta & -\cos2\theta 
\end{pmatrix} \right]\qquad\text{with}\qquad q = 0, 1, 
\end{equation}
where $\mat I$ is the $2 \times 2$ identity matrix and $\mathscr S(\theta)$ is a stability function. When $q = 0$, $\mat B_0(\theta)$ is a symmetrix 
matrix. If $\gamma(\theta) = \gamma(\pi + \theta)$, under some conditions on the stability function $\mathscr S(\theta)$, one can demonstrate that $\mat B_0(\theta)$ is a positive definite matrix  \cite{bao2023symmetrized}. When $q = 0$, $\mat B_0$ is an asymmetric matrix. From \cite{bao2023symmetrized, li2024structure}, we have 
\begin{equation}
-\left[ \gamma(\theta) + \gamma''(\theta) \right]r\kappa\vec n - \left[\gamma(\theta)z_s + \gamma'(\theta)r_s\right]\vec n = -\left[ r\mat B_q(\theta)\vec X_s \right]_s + \gamma(\theta)\vec e_1. 
\end{equation}
Additionally, by using $\kappa_m = \kappa(\mu_\mS - \kappa)$ and $\kappa = \vec X_{ss}\cdot\vec n$, and thanks to 
\begin{align}
\left[ (r(\mu_\mS)_s)_s - 2r\,\mu_\mS\,\kappa_m + \frac{1}{2}r\,\mu_\mS^3 \right]\vec n 
&=  (r(\mu_\mS)_s\vec n)_s - r(\mu_\mS)_s\,\kappa\,\vec X_s + r\mu_\mS(\mu_\mS)_s\vec X_s + \frac{1}{2}r\mu_\mS^3\vec n + \frac{1}{2}\mu_\mS^2\,\vec e_1   \nonumber \\
&~~~~ + \kappa\,\mu_\mS(\vec n\cdot\vec e_1)\,\vec n - \kappa\,\mu_\mS(\vec \tau\cdot\vec e_1)\,\vec\tau - \vec n\cdot\vec e_1\,(\mu_\mS)_s\vec X_s - \frac{1}{2}\mu_\mS^2\,\vec e_1 + \mu_\mS\,\kappa\,\vec e_1 \nonumber \\
&= \left[r(\mu_\mS)_s\vec n\right]_s + \frac{1}{2}\left[ r\,\mu_\mS^2\vec X_s \right]_s + \left[ \vec n\cdot\vec e_1\,\mu_\mS\,\vec X_s \right]_s - \frac{1}{2}\mu_\mS^2\vec e_1 + \mu_\mS\,\kappa\,\vec e_1 \nonumber \\
&= \left[ r(\mu_\mS)_s\,\vec n + \frac{1}{2}r\,\mu_\mS^2\vec X_s + \vec n\cdot\vec e_1\,\mu_\mS\vec X_s \right]_s - \frac{1}{2}\mu_\mS^2\,\vec e_1 + \mu_\mS\,\kappa\,\vec e_1, 
\end{align}
we can obtain 
\begin{align} \label{new1}
r\mu\vec n 
&= -\left[ \gamma(\theta) + \gamma''(\theta) \right]r\,\kappa\,\vec n - \left[ \gamma(\theta)z_s + \gamma'(\theta)r_s \right]\vec n + \varepsilon ^2\left[ (r(\mu_\mS)_s)_s - 2r\,\mu_\mS\,\kappa_m + \frac{1}{2}r\,\mu_\mS^3 \right]\vec n \nonumber \\
&= -\left[ r\mat B_q(\theta)\vec X_s \right]_s + \varepsilon ^2\left[ r(\mu_\mS)_s\,\vec n + \frac{1}{2}r\,\mu_\mS^2\vec X_s + \vec n\cdot\vec e_1\,\mu_\mS\vec X_s \right]_s + \gamma(\theta)\vec e_1 - \frac{1}{2}\varepsilon ^2\,\mu_\mS^2\vec e_1 + \varepsilon ^2\,\mu_\mS\,\kappa\,\vec e_1. 
\end{align}
Noting the mean curvature $\mu_\mS = \kappa - \frac{\vec n\cdot\vec e_1}{r}$, by differentiating with respect to variable $t$, we have 
\begin{align} \label{piankappa1}
r_t\,\mu_\mS + r\,(\mu_\mS)_t = r_t\kappa + r\kappa_t - (\vec n\cdot\vec e_1)_t.
\end{align}
From \cite{bao2024energy}, we can obtain the relations
\begin{align}
&\kappa_t = \left[ \vec n\cdot\left( \vec X_t \right)_s \right]_s - \left[ \vec X_s\cdot\left( \vec X_t \right)_s \right]\kappa,  \label{piankappa2} \\
&\left( \vec X_s \right)_t = \left( \vec X_t \right)_s - \left[ \vec X_s\cdot\left( \vec X_t \right)_s \right]\vec X_s = \left[ \vec n\cdot\left( \vec X_t \right)_s \right]\vec n. \label{piankappa3}
\end{align}
Therefore, we have 
\begin{align}\label{piankappa4}
\vec n_t = \left[ -\left(\vec X_s\right)^\bot \right]_t = -\left[ \left(\vec X^\bot\right)_s \right]_t = -\left[ \left( \vec X^\bot \right)_t \right]_s + \left[ \left( \vec X^\bot \right)_s\cdot\left[ \left( \vec X^\bot \right)_t \right]_s \right]\left( \vec X^\bot\right)_s = -\left[ \vec X_s\cdot\left[ \left( \vec X^\bot \right)_t \right]_s \right]\vec X_s = -\left[ \vec n\cdot\left( \vec X_t \right)_s \right]\vec X_s. 
\end{align}
Taking \eqref{piankappa2}, \eqref{piankappa3} and \eqref{piankappa4} into \eqref{piankappa1}, we have
\begin{align} \label{new2}
r_t\,\mu_\mS + r\,(\mu_\mS)_t = \left[ r\left( \vec n\cdot\left( \vec X_t \right)_s \right) \right]_s - \vec X_s\cdot\left( \vec X_t \right)_s\left[ r\,\mu_\mS + \vec n\cdot\vec e_1
 \right] + r_t\,\kappa. 
\end{align}
Then by combining \eqref{new1} and \eqref{new2}, we can obtain a new geometric PDE as follows: 
\begin{subequations}
\label{eqn:newPDE}
\begin{align}
&r\,\vec X_t\cdot\vec n = (r\,\mu)_s, \\
&r\,\mu\,\vec n = \left[ -r\mat B_q(\theta)\vec X_s  + \varepsilon ^2 \left( r(\mu_\mS)_s\,\vec n + \frac{1}{2}r\,\mu_\mS^2\vec X_s + \vec n\cdot\vec e_1\,\mu_\mS\vec X_s \right) \right]_s + \gamma(\theta)\vec e_1 - \frac{1}{2}\varepsilon ^2\,\mu_\mS^2\vec e_1 + \varepsilon ^2\,\mu_\mS\,\kappa\,\vec e_1, \\
&r_t\,\mu_\mS + r\,(\mu_\mS)_t = \left[ r\left( \vec n\cdot\left( \vec X_t \right)_s \right) \right]_s - \vec X_s\cdot\left( \vec X_t \right)_s\left[ r\,\mu_\mS + \vec n\cdot\vec e_1
 \right] + r_t\,\kappa, \\
&r\,\mu_\mS = r\,\kappa - \vec n\cdot\vec e_1, 
\end{align}  
\end{subequations}
with the boundary conditions (i)-(iv) given in \eqref{eqn:boundry1}-\eqref{eqn:boundry4}.
From \cite{barrett2021stable}, there hold 
\begin{align}
\vec n(\rho, t)\cdot\vec e_1 = 0, \qquad \vec X_\rho(\rho, t)\cdot\vec e_2 = 0 \quad \text{for} \quad (\rho, t)\in\{0, 1\}\times[0, T], 
\end{align}
and  
\begin{align}
\lim_{\rho \to \rho_0}\frac{\vec n(\rho, t)\cdot\vec e_1}{r(\rho, t)} = \lim_{\rho \to \rho_0}\frac{\vec n_\rho(\rho, t)\cdot\vec e_1}{r_\rho(\rho, t)} = \vec n_s(\rho_0, t)\cdot\vec\tau(\rho_0, t) = -\kappa(\rho_0, t)\quad \text{for}\quad (\rho_0, t)\in \{ 0, 1 \}\times [0, T]. 
\end{align}
Except the initial condition \eqref{initial}, we also need know the initial value of $\mu_\mS(\rho, \cdot)$, which can be computed by
\begin{align}\label{eqn:initial}
\mu_\mS(\rho, 0)= \left\{\begin{matrix}
 2\kappa(\rho, 0) & \quad \text{for}~\rho\in \{ 0, 1 \},  \\
 \kappa(\rho, 0) - \frac{\vec n\cdot\vec e_1}{r(\rho, 0)} & \quad \text{for}~\rho\in (0,1). 
\end{matrix}\right.
\end{align}

\section{Variational formulation and its properties}\label{sec4}

In this section, we build the variational formulation of the new geometric PDE \eqref{eqn:newPDE}, and then demonstrate the volume conservation and energy decay properties of the variational formulation. In order to introduce the variational formulation, we define a functional space on the domain $\mathbb I$ as 
\begin{align}
L^2(\mathbb I) := \left \{ u : \mathbb I\to\mathbb R \bigg| \int_{\Gamma(t)} \left| u(s) \right|^2 ds = \int_{\mathbb I} \left| u(s(\rho, t)) \right|^2 s_\rho\, d\rho < +\infty \right \}, 
\end{align}
equipped with the $L^2$-inner product 
\begin{align}
(u, v) := \int_{\Gamma(t)}u(s)v(s) ds = \int_{\mathbb I}u(s(\rho, t))v(s(\rho, t))\,s_{\rho}\, d\rho, \qquad \forall u, v \in L^2(\mathbb I). 
\end{align}
This $L^2$-inner product can also be directly extended to $[L^2(\mathbb I)]^2$. Moreover, we define the Sobolve spaces: 
\begin{align}
&\mat H^1(\mathbb I) := \left\{ u : \mathbb I \to \mathbb R, u\in L^2(\mathbb I) ~~\text{and}~~ u_\rho\in L^2(\mathbb I) \right\}, \\ 
&\mat H^1_0(\mathbb I) := \left\{ u\in \mat H^1(\mathbb I) : u(0) = u(1) = 0\right\}, 
\end{align}
and two another functional spaces: 
\begin{align}
&\mat H^{(r)}_{a, b}(\mathbb I) = \left\{ u\in\mat H^1(\mathbb I): u(0) = a; ~~ u(1) = b \right\}, \\
&\mat H^{(z)}_{a, b}(\mathbb I) = \left\{ u\in\mat H^1: u(1) = 0; ~~\text{if}~~ a>0, ~~ u(0) = 0 \right\}, 
\end{align}
where $a$ and $b$ are the radii of the inner and outer contact lines. If $a = b = 0$, we can directly obtain $\mat H^{(r)}_{a, b}(\mathbb I) = \mat H^{(r)}_{0, 0}(\mathbb I) = \mat H^1_0(\mathbb I)$. 

Then, by multiplying test functions $\phi \in \mat H^1(\mathbb I)$, $\vec\omega\in \mat H^1(\mathbb I)\times \mat H_0^1(\mathbb I)$, $\psi \in \mat H_0^1(\mathbb I)$ and $\chi \in L^2(\mathbb I)$ in \eqref{eqn:newPDE}, then integrating by parts, combining boundary conditions \eqref{eqn:boundry1}-\eqref{eqn:boundry4}, we can derive the variational formulation of the geometric PDE \eqref{eqn:newPDE}: given the initial curve open curve $\Gamma(0) := \vec X(\rho, 0)$, to find the solution $(\vec X(\cdot, t), \mu(\cdot, t), \mu_\mS(\cdot, t), \kappa(\cdot, t) )\in(\mat H_{a, b}^{(r)}(\mathbb I) \times \mat H_{a, b}^{(z)}(\mathbb I), \mat H^1(\mathbb I), \mat H_0^1(\mathbb I), L(\mathbb I))$, such that 
\begin{subequations}\label{bianfen}
\begin{align}
&\left( r\,\vec X_t\cdot\vec n, \phi \left| \vec X_\rho \right| \right) + \left( r\,\mu_\rho, \phi_\rho \left| \vec X_\rho \right|^{-1} \right) = 0, \qquad \forall\phi\in\mat H^1(\mathbb I), \label{bianfen1}\\
&\left( r\,\mu\,\vec n, \vec\omega \left| \vec X_\rho \right| \right) - \left( r\,\mat B_q(\theta)\,\vec X_\rho, \vec\omega_\rho \left| \vec X_\rho \right|^{-1} \right) + \varepsilon ^2\left( r\,(\mu_\mS)_\rho\,\vec n + \frac{1}{2}r\,\mu_\mS^2\,\vec X_\rho + \vec n\cdot\vec e_1\,\mu_\mS\,\vec X_\rho, \vec\omega_\rho\left| \vec X_\rho \right|^{-1} \right) \nonumber \\
&~~~~ - \left( \gamma(\theta)\vec e_1, \vec\omega\left| \vec X_\rho \right| \right) + \frac{\varepsilon ^2}{2}\left( \mu_\mS^2\vec e_1, \vec\omega\left| \vec X_\rho \right| \right) - \varepsilon ^2\left( \mu_\mS\,\kappa\,\vec e_1, \vec\omega\left| \vec X_\rho \right| \right) \nonumber \\
&~~~~ + \frac{1}{\eta}\left[ r_i\,r_t(0, t)\omega_1(0) + r_o\,r_t(1, t)\omega_1(1) \right] - \sigma\left[ r_o\omega_1(1) - r_i\omega_1(0) \right] = 0, \qquad \forall\vec\omega\in\mat H^1(\mathbb I)\times\mat H_0^1(\mathbb I), \label{bianfen2} \\ 
&\left( r_t\,\mu_\mS, \psi\left| \vec X_\rho \right| \right) + \left( r\,(\mu_\mS)_t, \psi\left| \vec X_\rho \right| \right) + \left( r\,\left( \vec n\cdot\left( \vec X_t \right)_\rho \right), \psi_\rho\left| \vec X_\rho \right|^{-1} \right) + \left( r\,\left( \vec X_\rho\cdot\left( \vec X_t \right)_\rho\right)\mu_\mS, \psi\left| \vec X_\rho \right|^{-1}\right) \nonumber \\
&~~~~ + \left( \vec n\cdot\vec e_1\,\left( \vec X_\rho\cdot\left( \vec X_t \right)_\rho\right), \psi\left| \vec X_\rho \right|^{-1}\right) - \left( r_t\,\kappa, \psi\left| \vec X_\rho \right| \right) = 0, \qquad \forall \psi\in\mat H_0^1(\mathbb I), \label{bianfen3} \\
&\left( r\,\mu_\mS, \chi\left| \vec X_\rho \right| \right) - \left( r\kappa, \chi\left| \vec X_\rho \right| \right) + \left( \vec n\cdot\vec e_1, \chi\left| \vec X_\rho \right| \right) = 0, \qquad \forall \chi\in L^2(\mathbb I), \label{bianfen4}
\end{align}
\end{subequations}
with the initial conditions \eqref{initial} and \eqref{eqn:initial}.

For the variational formulation \eqref{bianfen}, we can obtain its volume conservation and energy decay properties. 
\begin{thm}
(Volume conservation and energy decay). Let $(\vec X(\cdot, t), \mu(\cdot, t), \mu_\mS(\cdot, t), \kappa(\cdot, t) )$ be the solution of variational formulation \eqref{bianfen} with initial curve $\Gamma_0$. Then we have
\begin{align}\label{VE}
\text{vol}(\vec X(t)) = \text{vol}(\vec X(0)), \qquad W(t_2)\le W(t_1)\le W(0), \qquad t_2\ge t_1\ge 0. 
\end{align}
\begin{proof}
By directly differentiating $\text{vol}(\vec X(t))$ with respect to the time variable $t$, we have 
\begin{align}
\frac{\text{d}}{\text{d}t}\text{vol}(\vec X(t)) &= 2\pi\frac{\text{d}}{\text{d}t}\int_{\mathbb I} r\,r_\rho \, zd\rho \nonumber \\
&= 2\pi\int_{\mathbb I}( r_t\,r_\rho\, z + r\, r_\rho\, z_t + r\, (r_t)_\rho\, z )d\rho \nonumber \\
&= 2\pi\int_{\mathbb I}( r\,r_\rho\,z_t - r\,z_\rho\,r_t )d\rho \nonumber \\
&= 2\pi\int_{\mathbb I}r\,\vec X_t\cdot\vec n \left| \vec X_\rho 
\right| d\rho \nonumber \\
&= 2\pi \left( r\,\vec X_t\cdot\vec n, \left| \vec X_\rho 
\right| \right), \qquad t\ge 0. 
\end{align}
Selecting the test function $\phi = 1$ in \eqref{bianfen1}, we have 
\begin{align}
\left( r\,\vec X_t\cdot\vec n, \left| \vec X_\rho \right| \right) = \left( r\,\mu_\rho, 0\,\left| \vec X_\rho \right|^{-1} \right) = 0, \qquad t\ge 0, 
\end{align}
which implies volume conservation property in \eqref{VE}.

Furthermore, differentiating $W(t)$ with respect to $t$ and integrating by parts, it follows that
\begin{align}
\frac{1}{2\pi}\frac{\text{d}}{\text{d}t}W(t) &= \frac{\text{d}}{\text{d}t}\left(\int_0^{L(t)}r\gamma(\theta)ds + \varepsilon ^2\pi\int_0^{L(t)}r\mu_\mS^2ds - \frac{\sigma}{2}(r_o^2 - r_i^2)\right) \nonumber \\
&= \frac{\text{d}}{\text{d}t}\left(\int_{\mathbb I}r\,\gamma(\theta)\, s_\rho d\rho + \varepsilon ^2\pi\int_{\mathbb I}r\,\mu_\mS^2\, s_\rho d\rho - \frac{\sigma}{2}(r_o^2 - r_i^2)\right) \nonumber \\
&= \int_{\mathbb I}\left( r_t\,\gamma(\theta)\, s_\rho + r\,\gamma'(\theta)\, \theta_t\, s_\rho + r\,\gamma(\theta)\, (s_\rho)_t \right)d\rho \nonumber \\
&~~~~ + \varepsilon ^2\int_{\mathbb I}\left( r_t\,\mu_\mS^2\,s_\rho + 2r\,\mu_\mS\,(\mu_\mS)_t\,s_\rho + r\,\mu_\mS^2\,(s_\rho)_t \right)d\rho -\sigma\left[ r_o\,(r_o)_t - r_i\,(r_i)_t \right] \nonumber \\
&= \int_{\Gamma(t)}r\,\vec X_t\cdot\vec n\left[ -(\gamma(\theta) + \gamma''(\theta))\kappa - \frac{\gamma(\theta)\, z_s + \gamma'(\theta)\, r_s}{r} + \varepsilon ^2\frac{ (r\,(\mu_\mS)_s)_s -2r\,\mu_\mS\,\kappa_m + \frac{1}{2}r\,\mu_\mS^3}{r} \right]ds \nonumber \\
&~~~~ - \frac{1}{\eta}\left[ r_o(r_o)_t^2 + r_i(r_i)_t^2 \right] \nonumber \\
&= \int_{\Gamma(t)}r\,\vec X_t\cdot\vec n\,\mu ds - \frac{1}{\eta}\left[ r_o(r_o)_t^2 + r_i(r_i)_t^2 \right] \nonumber \\
&= \left( r\,\vec X_t\cdot\vec n, \mu\left| \vec X_\rho \right| \right) - \frac{1}{\eta}\left[ r_o(r_o)_t^2 + r_i(r_i)_t^2 \right]. \nonumber
\end{align}
Taking the test function $\phi = \mu$ in \eqref{bianfen1}, we can obtain
\begin{align}
\frac{1}{2\pi}\frac{\text{d}}{\text{d}t}W(t) = \left( r\,\vec X_t\cdot\vec n, \mu\left| \vec X_\rho \right| \right) - \frac{1}{\eta}\left[ r_o(r_o)_t^2 + r_i(r_i)_t^2 \right] = -\left( r\mu_\rho, \mu_\rho\left| \vec X_\rho \right|^{-1} \right) - \frac{1}{\eta}\left[ r_o(r_o)_t^2 + r_i(r_i)_t^2 \right] \le 0, \quad t\ge 0, \nonumber 
\end{align}
which implies energy decay property in \eqref{VE}. Therefore, we have completed this proof. 
\end{proof}
\end{thm}

\section{Parametric finite element approximation}\label{sec5}
In this section, we construct an energy-stable PFEM for the variational formulation \eqref{bianfen}. 
The time interval is divided as $[0, T] = \bigcup_{m=0}^{M-1} [t_m, t_{m+1}]$, with time step sizes given by $\Delta t_m = t_{m+1} - t_m$. The spatial domain $\mathbb{I} = \bigcup _{j=1}^{J}\mathbb I_j =  \bigcup_{j=1}^{J} [q_{j-1}, q_j]$ is uniformly partitioned into $J$ equal parts, with spatial step size $h = J^{-1}$. 
Then, we define the following finite element spaces:
\begin{align*}
&\mat V^h := \left\{ u\in C(\mathbb I): u\big|_{\mathbb I_j}\in\mathbb P_1, \quad \forall j=1, 2, \dots, J  \right\}\subseteq \mat H^1(\mathbb I), \\
&\mat V_0^h = \mat V^h\cap \mat H_0^1, \quad \mat V_{a, b}^{h, (r)} = \mat V^h\cap \mat H_{a, b}^{(r)}, \quad \mat V_{a, b}^{h, (z)} = \mat V^h\cap \mat H_{a, b}^{(z)}, 
\end{align*}
where $\mathbb P_1$ represents all polynomials with degrees at most $1$, $a$ and $b$ are two constants. We use $\Gamma^m := \vec X^m \in \mat V_{a, b}^{h, (r)}\times\mat V_{a, b}^{h, (z)}$ to approximate the evolution curve $\Gamma (t_m) := \vec X(\cdot, t_m)\in\mat H_{a, b}^{(r)}\times\mat H_{a, b}^{(z)}$. The approximation curve $\Gamma^m$ is comprised of line segments 
\begin{align}
\vec h_j^m := \vec X^m(\rho_j) - \vec X^m(\rho_{j-1}), \qquad j = 1, 2, \dots, J,  \nonumber 
\end{align}
and $| \vec h_j^m|$ representing the length of $\vec h_j^m$. Then, we can calculate the unit tangent vector $\vec\tau^m$ and the unit outward normal vector $\vec n^m$ of the approximate curve on interval $\mathbb I_j$ as 
\begin{align}
\vec\tau^m|_{\mathbb I_j} = \frac{\vec h_j^m}{\left| \vec h_j^m \right|} := \vec\tau_j^m, \qquad \vec n^m|_{\mathbb I_j} = -\frac{\left(\vec h_j^m\right)^\bot}{\left| \vec h_j^m \right|} := \vec n_j^m. \nonumber
\end{align}
Subsequently, we define the mass-lumped inner product on $\Gamma^m$ as follows 
\begin{align}
\left( \vec u, \vec v \right)_{\Gamma^m}^h := \frac{1}{2}\sum_{j=1}^{J}\left| \vec h_j^m \right|\left[ (\vec u\cdot\vec v)(\rho_j^-) + (\vec u\cdot\vec v)(\rho_{j-1}^+) \right], 
\end{align}
where $\vec v(\rho_j^{\pm}) = \lim\limits_{\rho \to \rho_j^{\pm}}\vec v(\rho) $ for $0\le j\le J$. 

By using the backward Euler method in terms of time, we establish an energy-stable parameter finite element approximation for the variational formulation \eqref{bianfen}: given initial data $(\vec X^0, \mu^0, \mu_\mS^0, \kappa^0) \in (\mat V_{a, b}^{h, (r)}\times\mat V_{a, b}^{h, (z)}, \mat V^h, \mat V_0^h, \mat V^h)$, find the solution $ (\vec X^{m+1}, \mu^{m+1}, \mu_\mS^{m+1}) \in (\mat V_{a, b}^{h, (r)}\times\mat V_{a, b}^{h, (z)}, \mat V^h, \mat V_0^h) $, such that
\begin{subequations} \label{numerapp}
\begin{align}
&\left( \frac{\vec X^{m+1} - \vec X^m}{\ttau}, \phi^h\vec f^{m + \frac{1}{2}} \right) + \left( r^m\,\mu_\rho^{m+1}, \phi_\rho^h\left| \vec X_\rho^m \right|^{-1} \right) = 0 \qquad \forall \phi^h\in\mat V^h, \label{numerapp1} \\
&\left( \mu^{m+1}\vec f^{m + \frac{1}{2}}, \vec\omega^h \right) - \left( r^m\mat B_q(\theta^m)\vec X_\rho^{m+1}, \vec\omega_\rho^h\left| \vec X_\rho^m \right|^{-1} \right) \nonumber \\
&~~~~ + \varepsilon ^2\left( r^m\,(\mu_\mS^{m+1})_\rho\,\vec n^m + \frac{1}{2}r^{m+1}\,(\mu_\mS^{m+1})^2\,\vec X_\rho^{m+1} + \vec n^m\cdot\vec e_1\,\mu_\mS^{m+1}\vec X_\rho^m, \vec\omega_\rho^h\left| \vec X_\rho^m \right|^{-1} \right) \nonumber \\
&~~~~ - \left( \gamma(\theta^{m+1})\vec e_1, \vec\omega^h\left| \vec X_\rho^{m+1} \right|  \right) + \frac{\varepsilon ^2}{2}\left( (\mu_\mS^{m+1})^2\vec e_1, \vec\omega^h\left| \vec X_\rho^m \right| \right) - \varepsilon ^2\left( \mu_\mS^{m+1}\,\kappa^m\,\vec e_1, \vec\omega^h\left| \vec X_\rho^m \right|\right) \nonumber \\
&~~~~ + \frac{1}{2\eta\ttau}\left[ (r_i^{m+1} + r_i^m)( r_i^{m+1} - r_i^m )\omega_1^h(0) + (r_o^{m+1} + r_o^m)( r_o^{m+1} - r_o^m )\omega_1^h(1) \right] \nonumber \\
&~~~~ - \frac{\sigma}{2} \left[ (r_o^{m+1} + r_o^m)\omega_1^h(1) - (r_i^{m+1} + r_i^m)\omega_1^h(0) \right] = 0 \qquad \forall \vec\omega^h = (\omega_1^h, \omega_2^h)\in \mat V^h\times\mat V_0^h, \label{numerapp2} \\
&\left( \frac{r^{m+1} - r^m}{\ttau}\mu_\mS^{m+1}, \psi^h \left| \vec X_\rho^m \right| \right) + \left( r^m\frac{\mu_\mS^{m+1} - \mu_\mS^m}{\ttau}, \psi^h \left| \vec X_\rho^m \right| \right) + \left( r^m\,\left( \vec n^m\cdot\frac{\vec X_\rho^{m+1} - \vec X_\rho^m}{\ttau} \right), \psi_\rho^h \left| \vec X_\rho^m \right|^{-1} \right) \nonumber \\
&~~~~ + \left( r^{m+1}\,\left( \vec X_\rho^{m+1}\cdot\frac{\vec X_\rho^{m+1} - \vec X_\rho^m}{\ttau}\right)\mu_\mS^{m+1}, \psi^h\left| \vec X_\rho^m \right|^{-1}  \right) + \left( \vec n^m\cdot e_1\,\left( \vec X_\rho^m\cdot\frac{ \vec X_\rho^{m+1} - \vec X_\rho^m}{\ttau} \right), \psi^h\left| \vec X_\rho^m \right|^{-1} \right) \nonumber \\
&~~~~ - \left( \frac{r^{m+1} - r^m}{\ttau}\,\kappa^m, \psi^h\left| \vec X_\rho^m \right| \right) = 0 \qquad \forall \psi^h\in\mat V_0^h, \label{numerapp3}
\end{align}
\end{subequations}
where $\vec f^{m + \frac{1}{2}}\in[\mat L^\infty(\mathbb I)]^2$ denotes the approximation of $\vec f = r\left| \vec X_\rho \right|\vec n$, given by
\begin{align}
\vec f^{m+\frac{1}{2}} = -\frac{1}{6}\left[ 2r^m\,\vec X_\rho^m +2r^{m+1}\,\vec X_\rho^{m+1} + r^m\,\vec X_\rho^{m+1} + r^{m+1}\,\vec X_\rho^m \right]^\bot. 
\end{align}

For \(\kappa^{m+1} \in \mathbb{V}^h\), it can be determined by solving the following equation:  
\begin{align}
\left( r^{m+1} \, \mu_\mathcal{S}^{m+1}, \, \chi^h \left| \vec{X}_\rho^{m+1} \right| \right) 
- \left( r^{m+1} \, \kappa^{m+1}, \, \chi^h \left| \vec{X}_\rho^{m+1} \right| \right) 
+ \left( \vec{n}^{m+1} \cdot \vec{e}_1, \, \chi^h \left| \vec{X}_\rho^{m+1} \right| \right) = 0 
\qquad \forall \, \chi^h \in \mathbb{V}^h.
\end{align}
Alternatively, \(\kappa^{m+1} \in \mathbb{V}^h\) can also be computed by solving:
\begin{align}
\left( \kappa^{m+1} \, \vec{n}^{m+1}, \, \vec{g}^h \left| \vec{X}_\rho^{m+1} \right| \right) 
- \left( \vec{X}_\rho^{m+1}, \, \vec{g}_\rho^h \left| \vec{X}_\rho^{m+1} \right|^{-1} \right) = 0 
\qquad \forall \, \vec{g}^h = (g_1^h, g_2^h) \in \mathbb{V}_0^h \times \mathbb{V}_0^h.
\end{align}

For the numerical approximation \eqref{numerapp}, we have the following energy stability result. To this end, we first introduce a lemma that is crucial to the proof of the energy stability. 
\begin{lem}\label{lem}
For the surface energy matrix $\mat B_q(\theta)$ with $q=0, 1$, the following cases are included: 
\begin{itemize}
    \item[\textbf{Case I.}] For $q=0$, if $\gamma(\theta)=\gamma(\pi+\theta)$ and 
    \begin{equation}
\mathscr S(\theta)\ge \mathscr S_0(\theta):=inf\left \{ \mathscr S(\theta)|\bigg[\gamma(\theta)\mat{B}_q(\theta)\,(\cos\hat{\theta}, \sin\hat{\theta})^\top\bigg]\cdot(\cos\hat{\theta}, \sin\hat{\theta})^\top\geq \gamma(\hat{\theta})^2, \forall\hat{\theta}\in[-\pi, \pi] \right \}, \quad\theta\in[-\pi, \pi], 
\end{equation}
then one can demonstrate that $\mat B_q(\theta)$ is a symmetric positive definite matrix and 
    \begin{align}
\mat B_q(\theta)\vec u\cdot(\vec u - \vec v) \ge \frac{1}{2}\mat B_q(\theta)\vec u\cdot \vec u - \frac{1}{2}\mat B_q(\theta)\vec v\cdot \vec v,  \qquad \forall \vec u,~\vec v\in \mathbb R^2,  
\end{align}
where $\theta = \arctan\frac{v_2}{v_1}$.
    \item[\textbf{Case II.}] For $q=1$, if $3\gamma(\theta)\ge \gamma(\pi + \theta)$ and 
   \begin{equation}
   \mathscr S(\theta)\ge \mathscr S_0(\theta):=inf\left \{ \varepsilon ^2\ge 0:P_\varepsilon ^2(\theta, \hat{\theta})-Q(\theta, \hat{\theta})\ge 0, \forall\hat{\theta}\in[-\pi, \pi] \right \}, \quad\theta\in[-\pi, \pi], 
   \end{equation}
   with $P_{\varepsilon ^2}(\theta, \hat{\theta})$ and $ Q(\theta, \hat{\theta})$ defined by  
   \begin{subequations}
   \begin{align}
   &P_{\varepsilon ^2}(\theta, \hat{\theta}):=2\sqrt{(\gamma(\theta)+\varepsilon ^2(-\sin\hat{\theta}\cos\theta+\cos\hat{\theta}\sin\theta)^2)\gamma(\theta)}, \quad\forall\theta, \hat{\theta}\in[-\pi, \pi], \quad\varepsilon ^2\ge 0, \\
   &Q(\theta, \hat{\theta}):=\gamma(\hat{\theta})+\gamma(\theta)(\sin\theta\sin\hat{\theta}+\cos\theta\cos\hat{\theta})-\gamma'(\theta)(-\sin\hat{\theta}\cos\theta+\cos\hat{\theta}\sin\theta), \quad\forall\theta, \hat{\theta}\in[-\pi, \pi]. 
  \end{align}
  \end{subequations}
   Then there holds 
   \begin{equation}\label{inequality}
   \frac{1}{\left | \vec v \right |}\left(\mat{B}_q(\theta)\vec w\right)\cdot(\vec w-\vec v)\ge\left | \vec w \right |\gamma(\hat{\theta})-\left | \vec v \right |\gamma(\theta), 
   \end{equation}
   where $(-\sin \theta, \cos \theta) = \frac{\vec v}{\left | \vec v \right |}$, $(-\sin \hat{\theta}, \cos \hat{\theta}) = \frac{\vec w}{\left | \vec w \right |}$.
\end{itemize}
\begin{itemize}
\item[\textbf{Case III.}] For $q = 0$, if $3\gamma(\theta)\ge \gamma(\pi + \theta)$ and 
   \begin{equation}
   \mathscr S(\theta)\ge \mathscr S_0(\theta):=inf\left \{ \varepsilon ^2\ge 0:P_\varepsilon ^2(\theta, \hat{\theta})-Q(\theta, \hat{\theta})\ge 0, \forall\hat{\theta}\in[-\pi, \pi] \right \}, \quad\theta\in[-\pi, \pi], 
   \end{equation}
   with $P_{\varepsilon ^2}(\theta, \hat{\theta})$ and $ Q(\theta, \hat{\theta})$ defined by  
   \begin{subequations}
   \begin{align}
   &P_{\varepsilon ^2}(\theta, \hat{\theta}):=2\sqrt{(-\gamma(\theta)+\varepsilon ^2(-\sin\hat{\theta}\cos\theta+\cos\hat{\theta}\sin\theta)^2 + f(\theta, \hat{\theta}))\gamma(\theta)}, \quad\forall\theta, \hat{\theta}\in[-\pi, \pi], \quad\varepsilon ^2\ge 0, \\
   &Q(\theta, \hat{\theta}):=\gamma(\hat{\theta})+\gamma(\theta)(\sin\theta\sin\hat{\theta}+\cos\theta\cos\hat{\theta}) -\gamma'(\theta)(-\sin\hat{\theta}\cos\theta+\cos\hat{\theta}\sin\theta), \quad\forall\theta, \hat{\theta}\in[-\pi, \pi], 
  \end{align}
  \end{subequations}
  where $f(\theta, \hat{\theta})$ is defined as follows
  \begin{align}
     f(\theta, \hat{\theta}) = 2(\sin\theta\sin\hat{\theta} + \cos\theta\cos\hat{\theta}) -\gamma'(\theta)(-\sin\hat{\theta}\cos\theta+\cos\hat{\theta}\sin\theta).  
  \end{align}
   Then there holds 
   \begin{equation}\label{inequality2}
   \frac{1}{\left | \vec v \right |}\left(\mat{B}_q(\theta)\vec w\right)\cdot(\vec w-\vec v)\ge\left | \vec w \right |\gamma(\hat{\theta})-\left | \vec v \right |\gamma(\theta), 
   \end{equation}
   where $(-\sin \theta, \cos \theta) = \frac{\vec v}{\left | \vec v \right |}$. 
\end{itemize}
\end{lem}
\begin{proof}
 We omit the proof here, as it follows similar approaches to those found in Refs. \cite{bao2023symmetrized, bao2024structure, bao2024structure1}.
\end{proof}
\begin{rem}
Due to the equivalence of $r\mu\vec n = -\left[ \gamma(\theta) + \gamma''(\theta) \right]r\kappa\vec n - \left[\gamma(\theta)z_s + \gamma'(\theta)r_s\right]\vec n$ and $r\mu\vec n = -\left[ \gamma(\theta) + \gamma''(\theta) \right]r\kappa\vec n - \left[\gamma(\theta)z_s + \gamma'(\theta)r_s\right]\vec n = -\left[ r\mat B_q(\theta)\vec X_s \right]_s + \gamma(\theta)\vec e_1$, we know that the continuous model, with the stability function in the matrix $\mat B_q(\theta)$, remains ill-posed. Therefore, from this perspective, it is necessary to introduce the Willmore regularization term, ensuring that the model is well-posed at the continuous level, thereby guaranteeing the well-posedness of the corresponding numerical method. 
\end{rem}

\begin{thm}\label{thm:ener}
(Energy stability) Let $(\vec X^{m+1}, \mu^{m+1}, \mu_\mS^{m+1}, \kappa^{m+1}) \in (\mat V_{a, b}^{h, (r)}\times\mat V_{a, b}^{h, (z)}, \mat V^h, \mat V_0^h, \mat V^h)$ be the numerical solution obtained from numerical approximation \eqref{numerapp}. Then, we can conclude that the total energy is unconditionally stable, i.e.
\begin{align}
W(\vec X^{m+1})\le W(\vec X^m)\le W(\vec X^0), \qquad 0\le m \le M-1. 
\end{align}
\end{thm}
\begin{proof}
The energy stability result holds for the \textbf{Case I}, \textbf{Case II} and \textbf{Case III}. However, for simplicity, we present the energy stability proof only for \textbf{Case III}. 

Selecting $\phi^h = \ttau\mu^{m+1}$ in \eqref{numerapp1}, $\vec\omega^h = \vec X^{m+1} - \vec X^m$ in \eqref{numerapp2}, and $\psi^h = \varepsilon ^2\ttau\,\mu_\mS^{m+1}$ in \eqref{numerapp3}, then  by rearranging these three expressions, we can obtain
\begin{align}\label{enerequa}
&\left( r^m\,\mat B_q(\theta^m)\vec X_\rho^{m+1}, \left( \vec X^{m+1} - \vec X^m \right)_\rho \left| \vec X_\rho^m \right|^{-1} \right) + \left( \gamma(\theta^{m+1}), (r^{m+1} - r^m)\left| \vec X_\rho^{m+1} \right| \right) \nonumber \\
&~~~~ + \frac{\varepsilon ^2}{2}\left( (r^{m+1} - r^m)\,\mu_\mS^{m+1}, \mu_\mS^{m+1}\left| \vec X_\rho^m \right| \right) + \varepsilon ^2\left( r^m\left( \mu_\mS^{m+1} - \mu_\mS^m \right), \mu_\mS^{m+1} \left| \vec X_\rho^m \right| \right) + \frac{\varepsilon ^2}{2}\left( r^{m+1}\,(\mu_\mS^{m+1})^2\vec X_\rho^{m+1}, \left( \vec X^{m+1} - \vec X^m \right)_\rho \left| \vec X_\rho^m \right|^{-1} \right) \nonumber \\
&~~~~ - \frac{\sigma}{2}\left( \left( r_o^{m+1} + r_o^m \right)\left( r_o^{m+1} - r_o^m \right) - \left( r_i^{m+1} + r_i^m \right)\left( r_i^{m+1} - r_i^m \right) \right) \nonumber \\
&~~~~ = -\left( r^m\,\mu_\rho^{m+1}, \mu_\rho^{m+1}\left| \vec X_\rho^m \right| \right) -\frac{1}{2\eta\ttau}\left[ \left( r_i^{m+1} + r_i^m \right)\left( r_i^{m+1} - r_i^m \right)^2 + \left( r_o^{m+1} + r_o^m \right)\left( r_o^{m+1} - r_o^m \right)^2 \right]. 
\end{align}
Then, by choosing $\vec w = \vec X_\rho^{m+1}$ and $\vec v = \vec X_\rho^m$ in \eqref{inequality}, we have
\begin{align} \label{inequalityzhu}
\mat B_q(\theta^m)\vec X_\rho^{m+1}\cdot\left( \vec X_\rho^{m+1} - \vec X_\rho^m \right)\left| \vec X_\rho^m \right|^{-1}\ge \gamma(\theta^{m+1})\left| \vec X_\rho^{m+1} \right| - \gamma(\theta^m)\left| \vec X_\rho^m \right|.
\end{align}
By utilizing the properties of matrix $\mat B_q(\theta^m)$ in \eqref{inequalityzhu} and the inequality $(a-b)a\ge \frac{1}{2}a^2 - \frac{1}{2}b^2$, we obtain
\begin{align}\label{enerinequa}
&\left( r^m\,\mat B_q(\theta^m)\vec X_\rho^{m+1}, \left( \vec X^{m+1} - \vec X^m \right)_\rho \left| \vec X_\rho^m \right|^{-1} \right) + \left( \gamma(\theta^{m+1}), (r^{m+1} - r^m)\left| \vec X_\rho^{m+1} \right| \right) \nonumber \\
&~~~~ + \frac{\varepsilon ^2}{2}\left( (r^{m+1} - r^m)\,\mu_\mS^{m+1}, \mu_\mS^{m+1}\left| \vec X_\rho^m \right| \right) + \varepsilon ^2\left( r^m\left( \mu_\mS^{m+1} - \mu_\mS^m \right), \mu_\mS^{m+1} \left| \vec X_\rho^m \right| \right) + \frac{\varepsilon ^2}{2}\left( r^{m+1}\,(\mu_\mS^{m+1})^2\vec X_\rho^{m+1}, \left( \vec X^{m+1} - \vec X^m \right)_\rho \left| \vec X_\rho^m \right|^{-1} \right) \nonumber \\
&~~~~ - \frac{\sigma}{2}\left( \left( r_o^{m+1} + r_o^m \right)\left( r_o^{m+1} - r_o^m \right) - \left( r_i^{m+1} + r_i^m \right)\left( r_i^{m+1} - r_i^m \right) \right) \nonumber \\
& \ge \left( r^{m+1}, \gamma(\theta^{m+1})\left| \vec X_\rho^{m+1} \right| \right) - \left( r^m, \gamma(\theta^m)\left| \vec X_\rho^m \right| \right) + \frac{\varepsilon ^2}{2}\left( (r^{m+1} - r^m)\,\mu_\mS^{m+1}, \mu_\mS^{m+1}\left| \vec X_\rho^m \right| \right) \nonumber \\
&~~~~ + \frac{\varepsilon ^2}{2}\left( r^m\,\left( (\mu_\mS^{m+1})^2 - (\mu_\mS^m)^2 \right), \left| \vec X_\rho^m \right| \right) +  \frac{\varepsilon ^2}{2}\left( r^{m+1}\,(\mu_\mS^{m+1})^2\vec X_\rho^{m+1}, \left( \vec X^{m+1} - \vec X^m \right)_\rho \left| \vec X_\rho^m \right|^{-1} \right) \nonumber \\
&~~~~ - \frac{\sigma}{2}\left( \left( r_o^{m+1} + r_o^m \right)\left( r_o^{m+1} - r_o^m \right) - \left( r_i^{m+1} + r_i^m \right)\left( r_i^{m+1} - r_i^m \right) \right) \nonumber \\
&= \left( r^{m+1}, \gamma(\theta^{m+1})\left| \vec X_\rho^{m+1} \right| \right) - \left( r^m, \gamma(\theta^m)\left| \vec X_\rho^m \right| \right) - \frac{\varepsilon ^2}{2}\left( r^m, (\mu_\mS^m)^2\left| \vec X_\rho^m \right| \right) \nonumber \\
&~~~~ + \frac{\varepsilon ^2}{2}\left( r^{m+1}\,(\mu_\mS^{m+1})^2, \left| \vec X_\rho^m \right| + \vec X_\rho^{m+1}\cdot\left( \vec X_\rho^{m+1} - \vec X_\rho^m
 \right)\,\left| \vec X_\rho^m \right|^{-1} \right) - \frac{\sigma}{2}\left( \left( r_o^{m+1} + r_o^m \right)\left( r_o^{m+1} - r_o^m \right) - \left( r_i^{m+1} + r_i^m \right)\left( r_i^{m+1} - r_i^m \right) \right) \nonumber \\
 &\ge \left( r^{m+1}, \gamma(\theta^{m+1})\left| \vec X_\rho^{m+1} \right| \right) + \frac{\varepsilon ^2}{2}\left( r^{m+1}\,(\mu_\mS^{m+1})^2, \frac{1}{2}\left( \left| \vec X^m \right| + \frac{\left| \vec X^{m+1} \right|^2}{\left| \vec X^m \right|} \right) \right) - \left( r^m, \gamma(\theta^m)\left| \vec X_\rho^m \right| \right) - \frac{\varepsilon ^2}{2}\left( r^m, (\mu_\mS^m)^2\left| \vec X_\rho^m \right| \right) \nonumber \\
 &~~~~ - \frac{\sigma}{2}\left( \left( r_o^{m+1} + r_o^m \right)\left( r_o^{m+1} - r_o^m \right) - \left( r_i^{m+1} + r_i^m \right)\left( r_i^{m+1} - r_i^m \right) \right) \nonumber \\
 &\ge \left( r^{m+1}, \gamma(\theta^{m+1})\left| \vec X_\rho^{m+1} \right| \right) + \frac{\varepsilon ^2}{2}\left( r^{m+1}\,(\mu_\mS^{m+1})^2, \left| \vec X_\rho^{m+1} \right| \right) - \frac{\sigma}{2}\left( (r_o^{m+1})^2 - (r_i^{m+1})^2 \right) \nonumber \\
 &~~~~ - \left( r^m, \gamma(\theta^m)\left| \vec X_\rho^m \right| \right) - \frac{\varepsilon ^2}{2}\left( r^m, (\mu_\mS^m)^2\left| \vec X_\rho^m \right| \right) + \frac{\sigma}{2}\left( (r_o^m)^2 - (r_i^m)^2 \right) = \frac{1}{2\pi}\left( W(\vec X^{m+1}) - W(\vec X^m) \right). 
\end{align}
Finally, combining \eqref{enerequa} and \eqref{enerinequa}, we can obtain
\begin{align}
W(\vec X^{m+1}) - W(\vec X^m)\le -2\pi\left( r^m\,\mu_\rho^{m+1}, \mu_\rho^{m+1}\left| \vec X_\rho^m \right| \right) -\frac{\pi}{\eta\ttau}\left[ \left( r_i^{m+1} + r_i^m \right)\left( r_i^{m+1} - r_i^m \right)^2 + \left( r_o^{m+1} + r_o^m \right)\left( r_o^{m+1} - r_o^m \right)^2 \right]\le 0, 
\end{align}
which implies the property of energy stability. Therefore, we have completed the proof.
\end{proof}
\begin{thm}\label{thm:vol}
(Volume conservation) Let $(\vec X^{m+1}, \mu^{m+1}, \mu_\mS^{m+1}, \kappa^{m+1}) \in (\mat V_{a, b}^{h, (r)}\times\mat V_{a, b}^{h, (z)}, \mat V^h, \mat V_0^h, \mat V^h)$ be the numerical solution obtained from the numerical approximation \eqref{numerapp}. Then we have
\begin{align}
\text{vol}(\vec X^{m+1}) - \text{vol}(\vec X^m) = 0,\qquad m=0, 1\ldots, M-1. 
\end{align}
\begin{proof}
By choosing $\phi^h = \ttau$ in \eqref{numerapp1} and recalling
\begin{align}
\text{vol}(\vec X^{m+1}) - \text{vol}(\vec X^{m}) = 2\pi\left( \vec X^{m+1} - \vec X^m, \vec f^{m+\frac{1}{2}} \right) \qquad \text{for} \qquad \vec X^{m+1},~\vec X^m\in \mat V_{a,b}^{h, (r)}\times \mat V_{a,b}^{h, (z)}, 
\end{align}
we directly derive the volume conservation. 
\end{proof}
\end{thm}
\begin{rem}
During the numerical tests, we employ the Newton-Raphson iteration to compute the implicit scheme \eqref{numerapp}. The iterative process is repeated until 
\begin{equation}
\left \| \vec X^{m+1, i+1} - \vec X^{m+1, i} \right \|_{L^\infty} + \left \| \mu^{m+1, i+1} - \mu^{m+1, i} \right \|_{L^\infty} + \left \| \mu_\mS^{m+1, i+1} - \mu_\mS^{m+1, i} \right \|_{L^\infty} \le \text{tol}. 
\end{equation}
Here, tol is a predefined tolerance level, ensuring that the iteratine process continues until the solution achieves a certain level of accuracy. 
\end{rem}

\section{Numerical results}\label{sec6}

In this section, we present numerical experiments to evaluate the performance of our proposed numerical schemes. These experiments validate the structure-preserving properties of the schemes, including volume conservation and energy stability, as well as their convergence results and 
mesh quality. Additionally, we simulate various processes of SSD to further demonstrate the applicability of the schemes.

To check the mesh quality and the volume conservation of the scheme, we define the mesh ratio $R^h(t)$ and loss volume $\Delta V$ at $t_m$ as follows 
\begin{equation} 
R^h(t)\big|_{t=t_m} := \frac{\text{max}_{1\le j \le J}\left| \vec X^m_j - X^m_{j-1}\right|}{\text{min}_{1\le j \le J}\left| \vec X^m_j - X^m_{j-1}\right|}, \quad \Delta V(t)\big|_{t=t_m} = \frac{\text{vol}(\vec X^m) - \text{vol}(\vec X^0)}{\text{vol}(\vec X^0)}, \quad
\overline{W}_c(t)|_{t = t_m} = W_c^m,  \quad m\ge 0. 
\end{equation}
In the numerical experiments, we choose the 4-fold anisotropy: $\gamma(\theta) = 1 + \beta\cos (4\theta)$ as the energy function, where $\beta$ represents degree of anisotropy. when $\beta = 0$, it denotes isotropic; when $0 < \beta\le 1/15$, it denotes weakly anisotropic; when $\beta\ge 1/15$, it denotes strongly anisotropic. During the numerical experiments, we choose the Newton-Raphson iteration to calculate the semi-implicit scheme \eqref{numerapp} and select the tolerance $\text{tol} = 1e - 8$. 

\textbf{Example 1} (Convergence tests) We test convergence by quantifying the difference between surfaces enclosed by the curves $\Gamma_1$ and $\Gamma_2$, using the manifold distance defined by
\begin{equation*}
\text{Md}(\Gamma_1, \Gamma_2):=\left |(\Omega_1\backslash\Omega_2)\cup(\Omega_2\backslash\Omega_1) \right |=\left |\Omega_1 \right |+\left |\Omega_2 \right |-2\left |\Omega_1\cap\Omega_2 \right |, \nonumber
\end{equation*}
with $\Omega_i$, $i=1, 2$ denoting the region enclosed by $\Gamma_i$, and $| \cdot |$ representing the area of the region. Let $\vec X^m$ denote numerical approximation of surface with mesh size $h$ and time step $\ttau$, then introduce approximate solution between interval $[t_m, t_{m+1}]$ as
\begin{equation}
\vec X_{h, \ttau}(\rho, t)=\frac{t-t_m}{\ttau}\vec X^m(\rho)+\frac{t_m-t}{\ttau}\vec X^{m+1}(\rho), \quad\rho\in\mathbb{I}. 
\end{equation}
We further define the errors by  
\begin{equation}
e_{h, \ttau}(t)=\text{Md}(\Gamma_{h, \ttau}, \Gamma_{\frac{h}{2}, \frac{\ttau}{4}}). 
\end{equation}
In this example, we choose the initial data $\vec X^0(\rho) = (10 + \cos(\pi\rho), \sin(\pi\rho))$. 
Figures \ref{fig:1}-\ref{fig:2} respectively illustrate the numerical errors and their orders for the structure-preserving method under various values of the anisotropic strength parameter $\beta$ and the regularization parameter $\varepsilon$.
From the figures, we find that the convergence rate with respect to the mesh size $h$ is of the second order, aligning with our expected results.

\begin{figure}[!htp]
\centering
\includegraphics[width=0.47\textwidth]{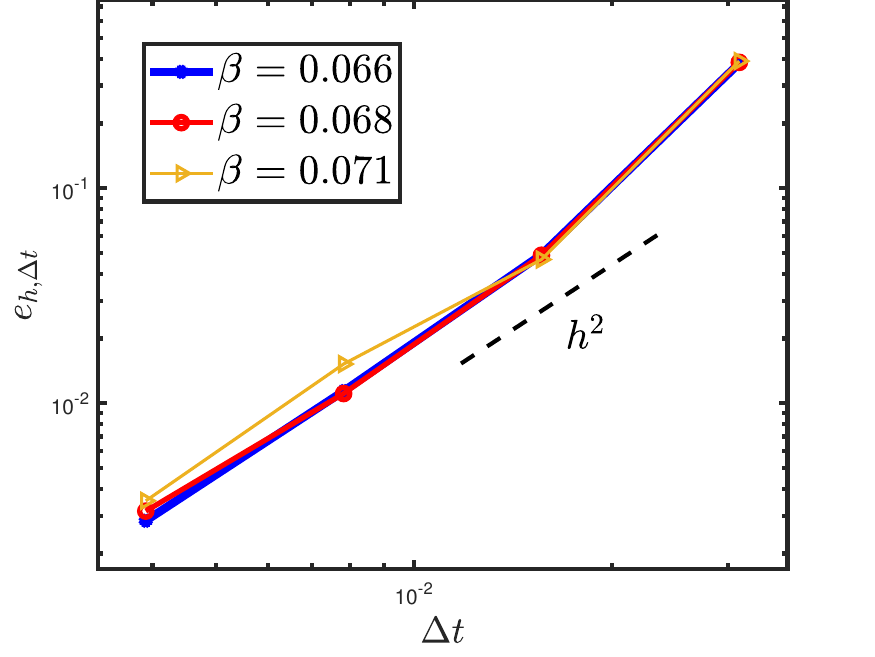}
\includegraphics[width=0.47\textwidth]{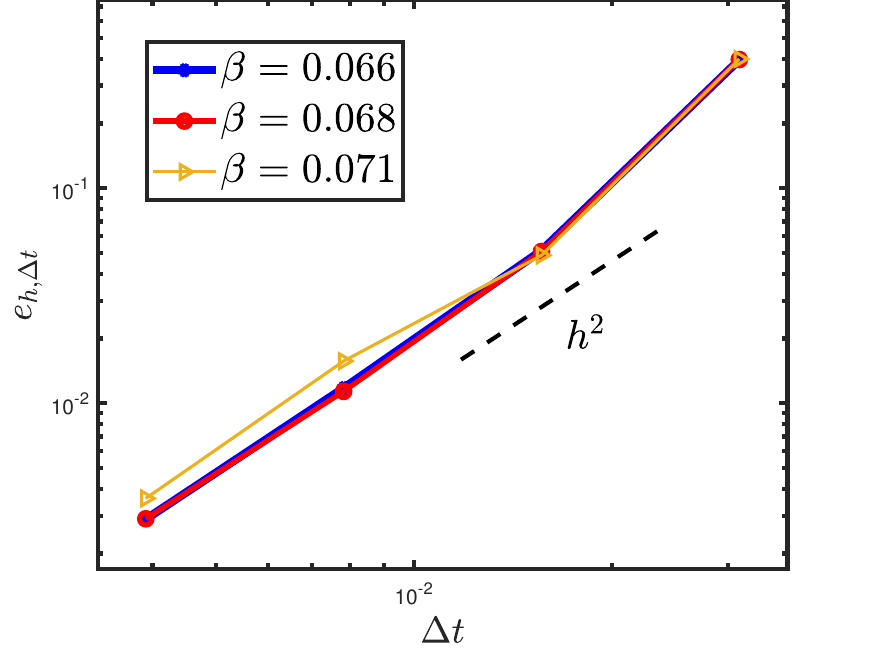}
\caption{Plot of the numberical errors at $\ttau_m = 1$(left panel) and $\ttau_m$ = 2(right panel) for 4-fold anisotropy. The initial data $\vec X^0(\rho) = (10 + \cos(\pi\rho), \sin(\pi\rho))$, and the parameters are selected as $\eta = 100$, $\sigma = -0.6$, $\varepsilon = 0.01$. }
\label{fig:1}
\end{figure}

\begin{figure}[!htp]
\centering
\includegraphics[width=0.47\textwidth]{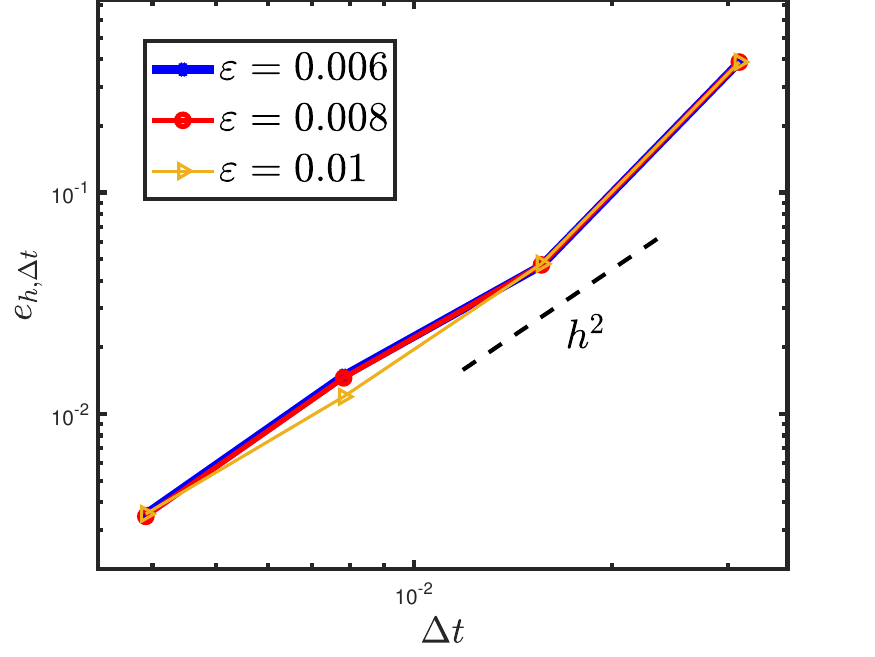}
\includegraphics[width=0.47\textwidth]{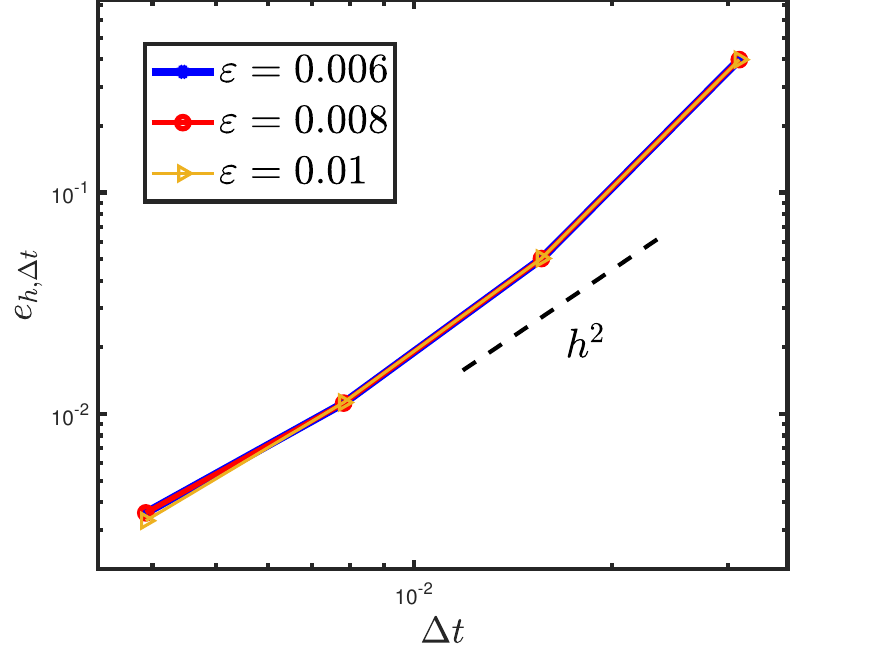}
\caption{Plot of numberical errors at $\ttau_m = 1$(left panel) and $\ttau_m$ = 2(right panel) for 4-fold anisotropy. The initial data $\vec X^0(\rho) = (10 + \cos(\pi\rho), \sin(\pi\rho))$and the parameters are selected as $\eta = 100$, $\sigma = -0.6$, $\beta = 0.07$. }
\label{fig:2}
\end{figure}

\textbf{Example 2} (Energy stability \& Volume conservation) In this example, we test the energy stability and volume conservation of the structure-preserving method \eqref{numerapp}. 
In the tests, two types of initial values are considered: 
 $\vec X^0 = (10 + \cos(\pi\rho), \sin(\pi\rho))$ and $\vec X^0 = ( \cos(\pi\rho/2), \sin(\pi\rho/2))$. 
For both types of boundary conditions, by varying the time step size \(\Delta t\), the anisotropic parameter \(\beta\), and the regularization parameter \(\epsilon\), we plot the energy ratio \(E(t)/E(0)\) of the structure-preserving method in Figures \ref{fig:3}-\ref{fig:4}. The results demonstrate energy stability across all cases, confirming the theoretical findings presented in Theorem \ref{thm:ener}. 
Subsequently, for the type of boundary condition $\vec X^0 = (10 + \cos(\pi\rho), \sin(\pi\rho))$ , by selecting different values of the parameter \(\beta\), we plot the volume change over time in Figure \ref{fig:5}. The figures clearly demonstrate volume conservation, consistent with Theorem \ref{thm:vol}.

\begin{figure}[!htp]
\centering
\includegraphics[width=0.45\textwidth]{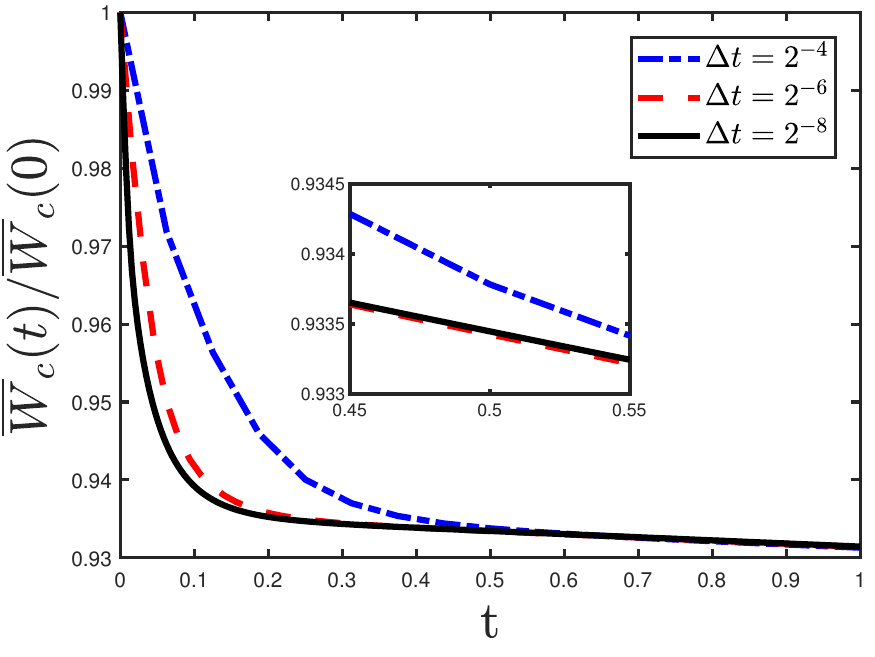}\hspace{0.5cm}
\includegraphics[width=0.45\textwidth]{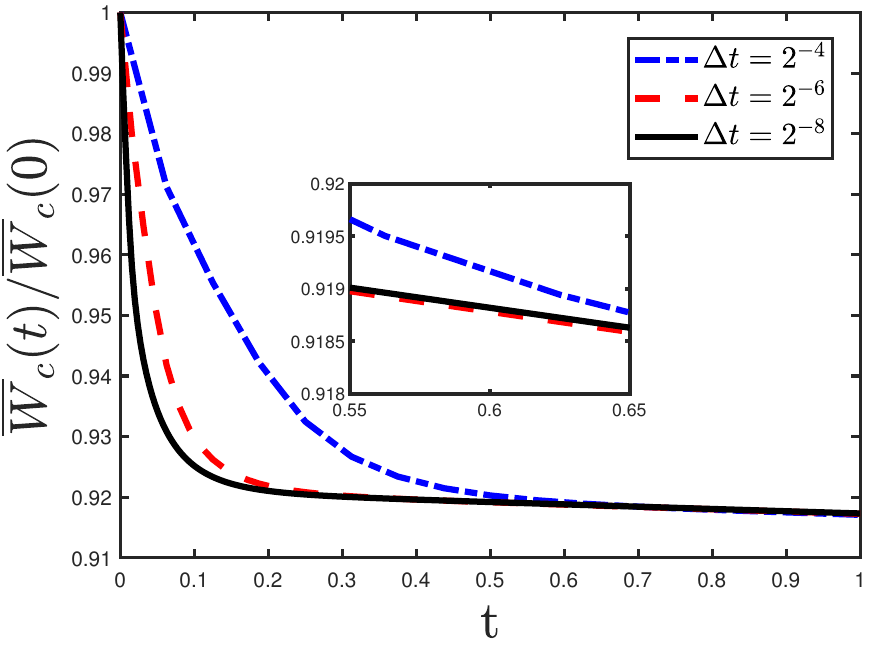}
\caption{The time history of the energy ratio $E(t)/E(0)$ employing structure-preserving method with $\beta = 0.07$ (left panel) and $\beta = 0.1$ (right panel). The initial data $\vec X^0(\rho) = (10 + \cos(\pi\rho), \sin(\pi\rho))$, and the parameters are selected as $\eta = 100$, $\sigma = -0.6$, $\varepsilon  = 0.01$. }
\label{fig:3}
\end{figure}

\begin{figure}[!htp]
\centering
\includegraphics[width=0.45\textwidth]{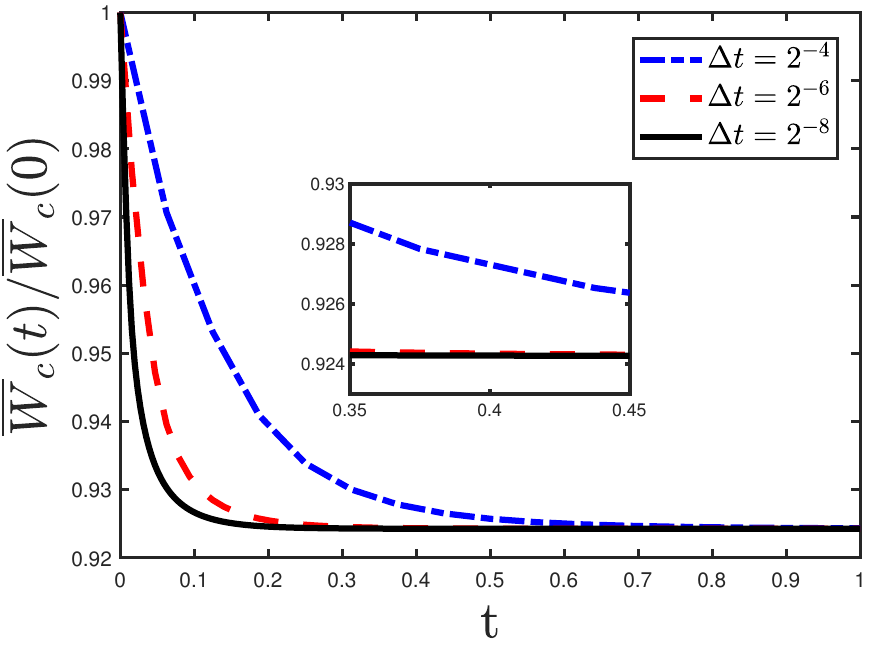}\hspace{0.5cm}
\includegraphics[width=0.45\textwidth]{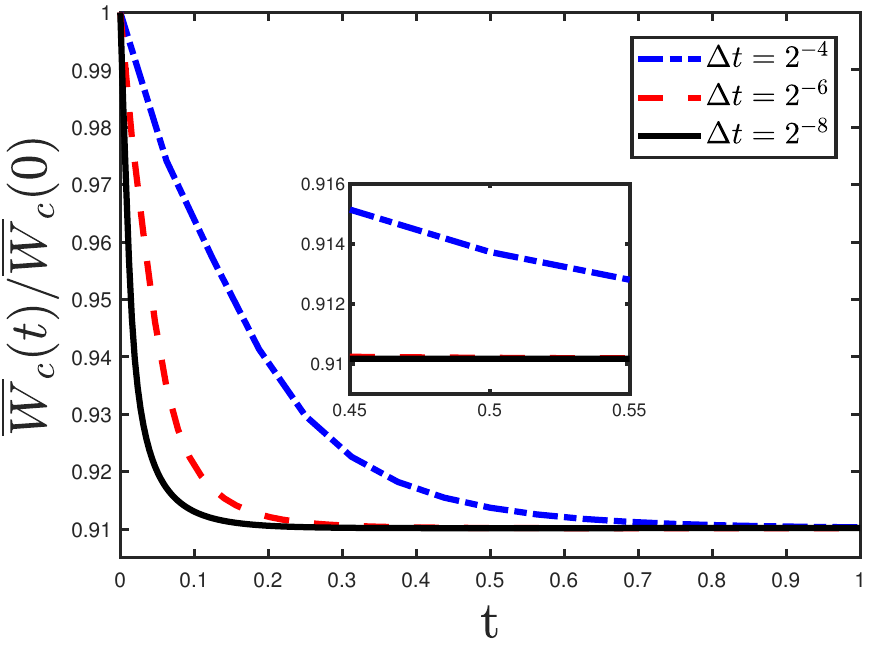}
\caption{The time history of the energy ratio $E(t)/E(0)$ employing structure-preserving method with $\beta = 0.07$ (left panel) and $\beta = 0.1$ (right panel). We choose the initial data $\vec X^0(\rho) = (\cos(\pi\rho/2), \sin(\pi\rho/2))$. The parameters are selected as $\eta = 100$, $\sigma = -0.6$, $\varepsilon  = 0.005$. }
\label{fig:4}
\end{figure}

\begin{figure}[!htp]
\centering
\includegraphics[width=0.45\textwidth]{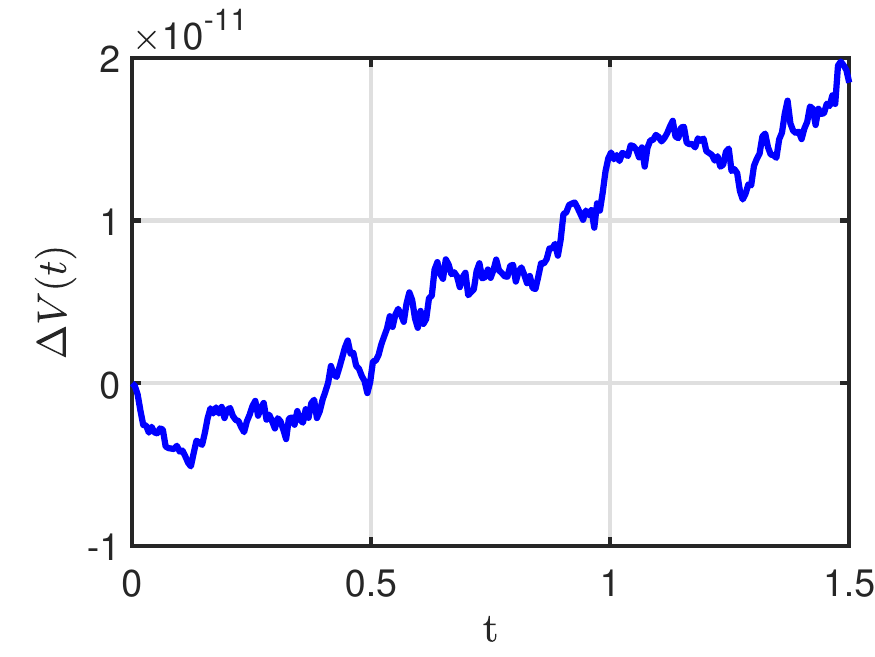}\hspace{0.5cm}
\includegraphics[width=0.45\textwidth]{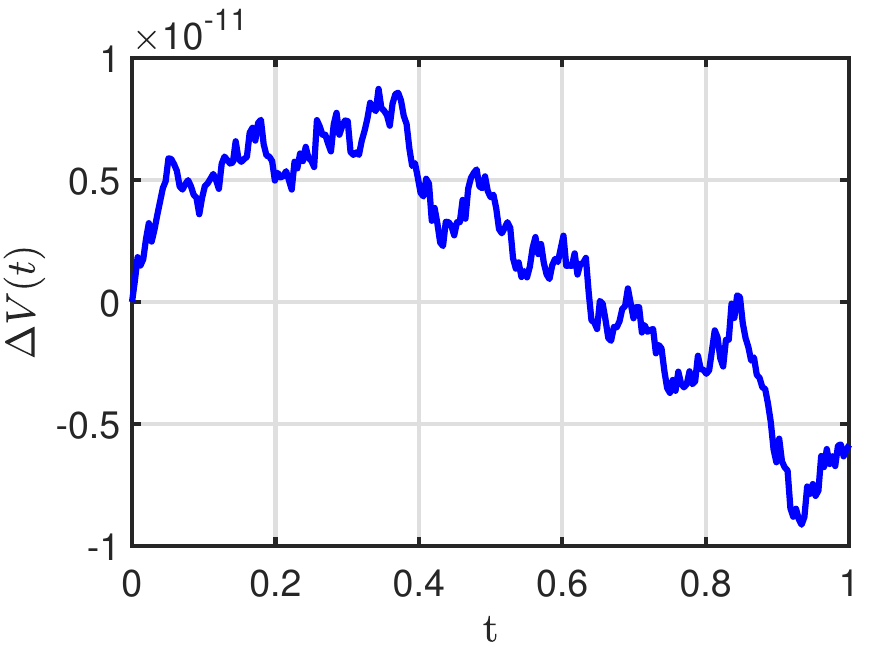}\hspace{0.5cm}
\includegraphics[width=0.45\textwidth]{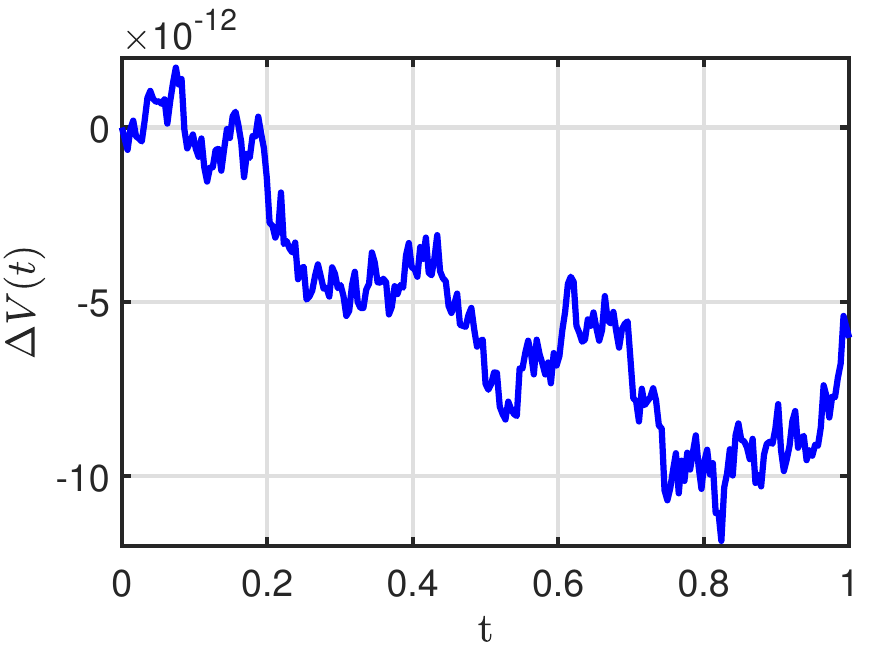}\hspace{0.5cm}
\includegraphics[width=0.45\textwidth]{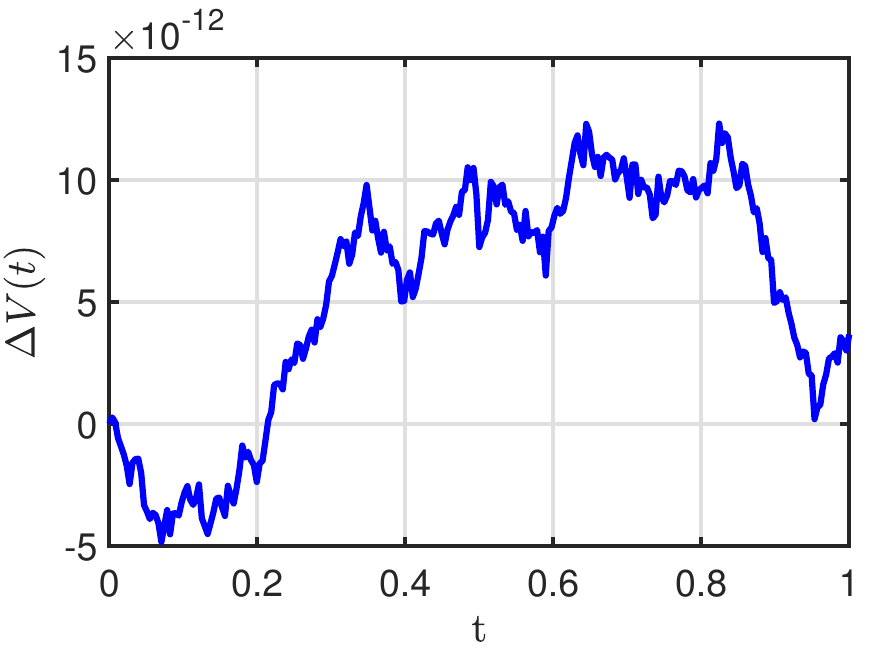}
\caption{The time history of the volume loss $\Delta V$. We choose the initial data $\vec X^0(\rho) = (10 + \cos(\pi\rho), \sin(\pi\rho))$. The parameters are selected as $\beta = 0.07, 0.08, 0.09, 0.1$ $\eta = 100$, $\sigma = -0.6$, $\varepsilon  = 0.01$, $J = 128$, $\ttau = 1/256$. }
\label{fig:5}
\end{figure}

\textbf{Example 3} (Mesh quality) In this example, our primary focus is on evaluating the mesh quality of the structure-preserving method throughout its evolution. The principal reason for presenting this example is to assess the influence of incorporating regularization terms on the quality of the mesh.
Throughout the tests, we utilize the same initial data and material parameters as used in Example 2. 
Figures \ref{fig:6}-\ref{fig:7} depict the ratio of the maximum to minimum mesh sizes throughout the evolution process, comparing scenarios both with and without the inclusion of regularization terms.
These numerical experiments show that adding the Willmore regularization term greatly enhances the mesh quality of the structure-preserving method, highlighting its importance.

\begin{figure}[!htp]
\centering
\includegraphics[width=0.45\textwidth]{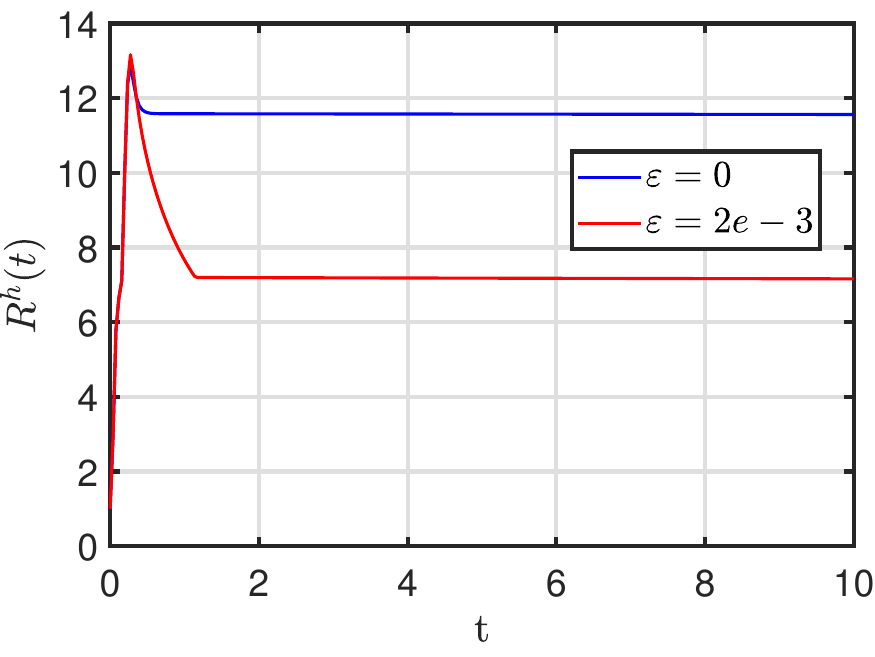}\hspace{0.5cm}
\includegraphics[width=0.45\textwidth]{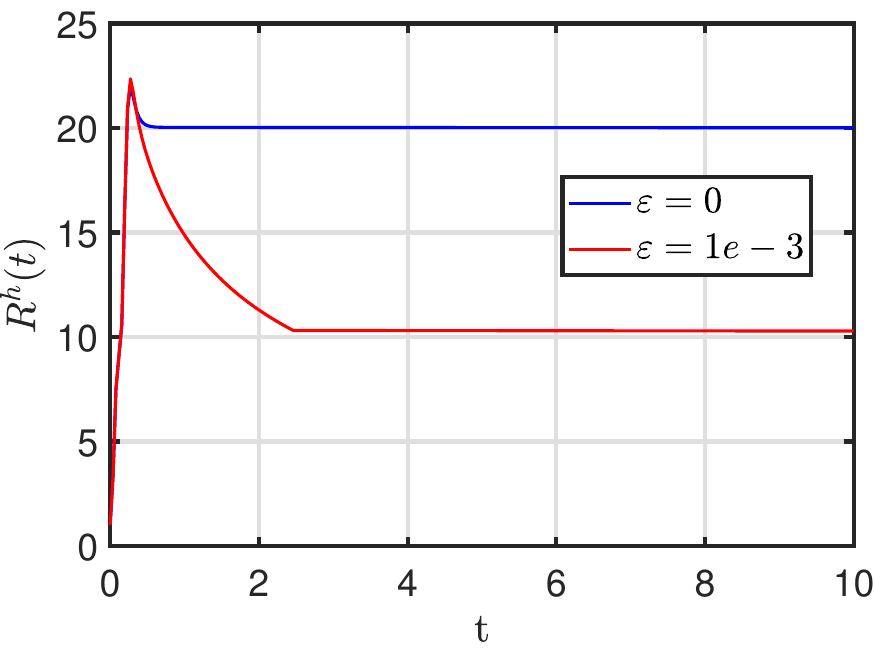}\hspace{0.5cm}
\includegraphics[width=0.45\textwidth]{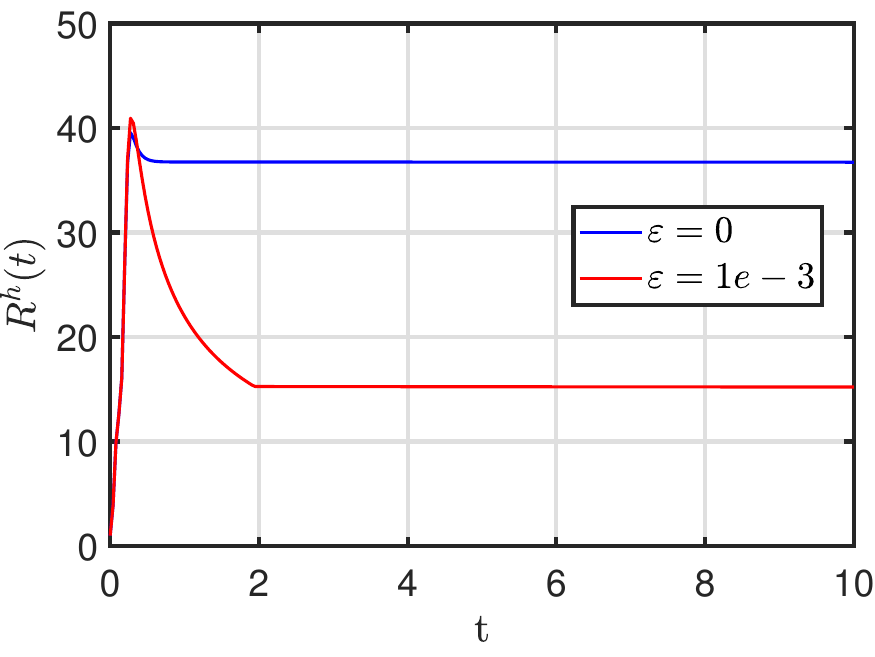}\hspace{0.5cm}
\includegraphics[width=0.45\textwidth]{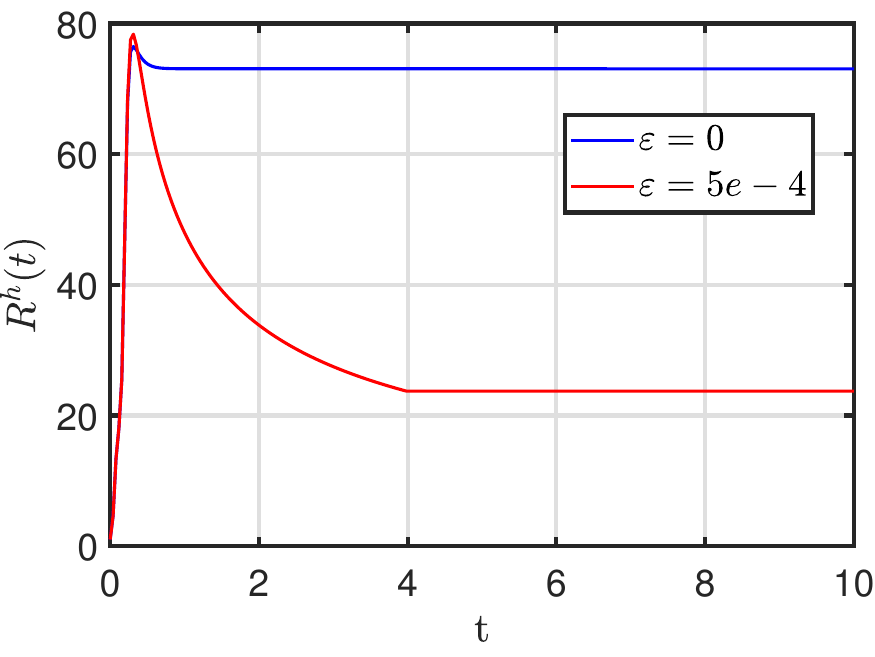}
\caption{Time evolution of the mesh ratio $R^h(t)$ for the four cases of $\beta = 0.35, 0.4, 0.45, 0.5$ . The initial data $\vec X^0(\rho) = (10 + \cos(\pi\rho), \sin(\pi\rho))$, and the parameters are selected as $\eta = 100$, $\sigma = -0.6$, $J = 65$, $\ttau = 5/128$. }
\label{fig:6}
\end{figure}

\begin{figure}[!htp]
\centering
\includegraphics[width=0.45\textwidth]{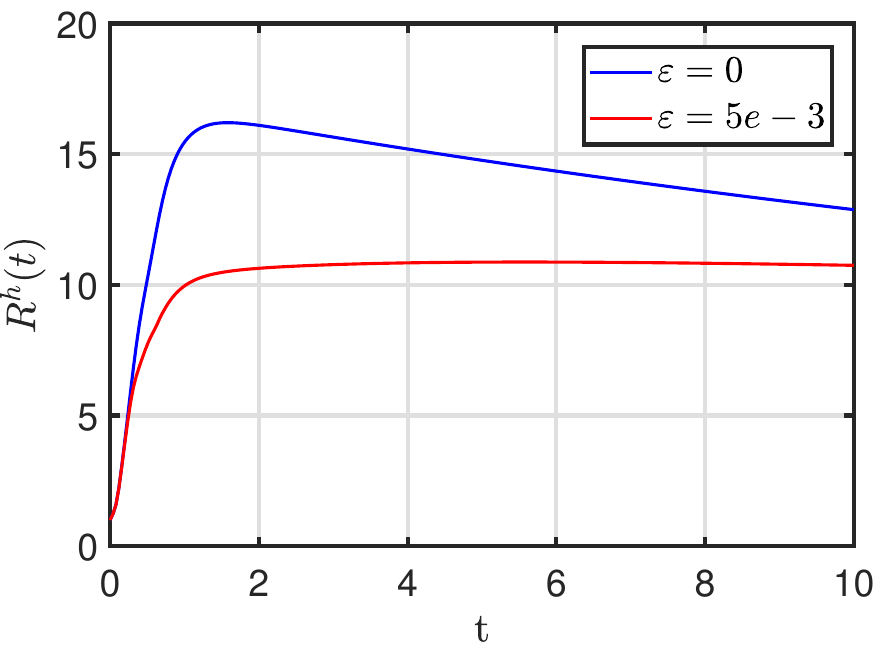}\hspace{0.5cm}
\includegraphics[width=0.45\textwidth]{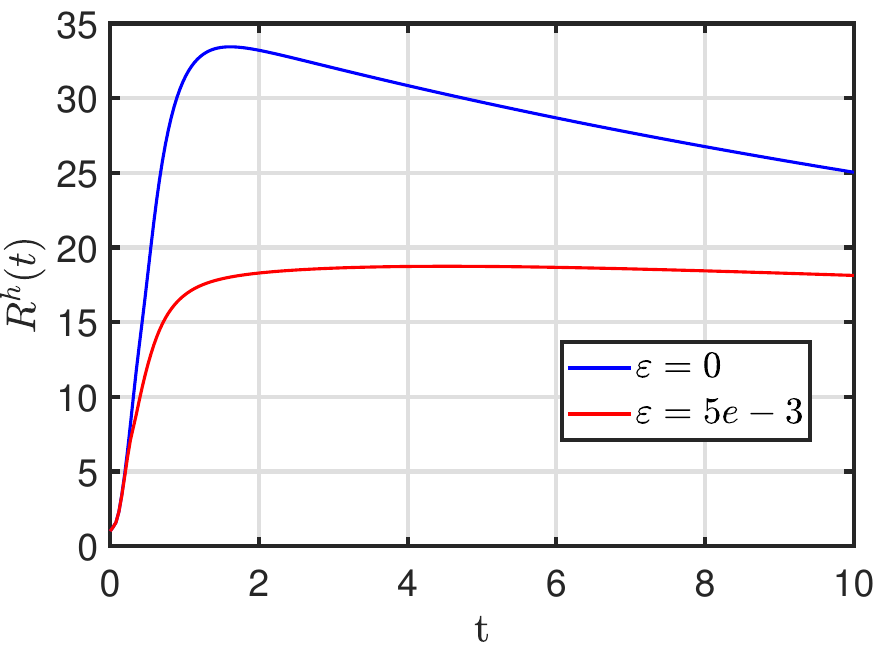}\hspace{0.5cm}
\includegraphics[width=0.45\textwidth]{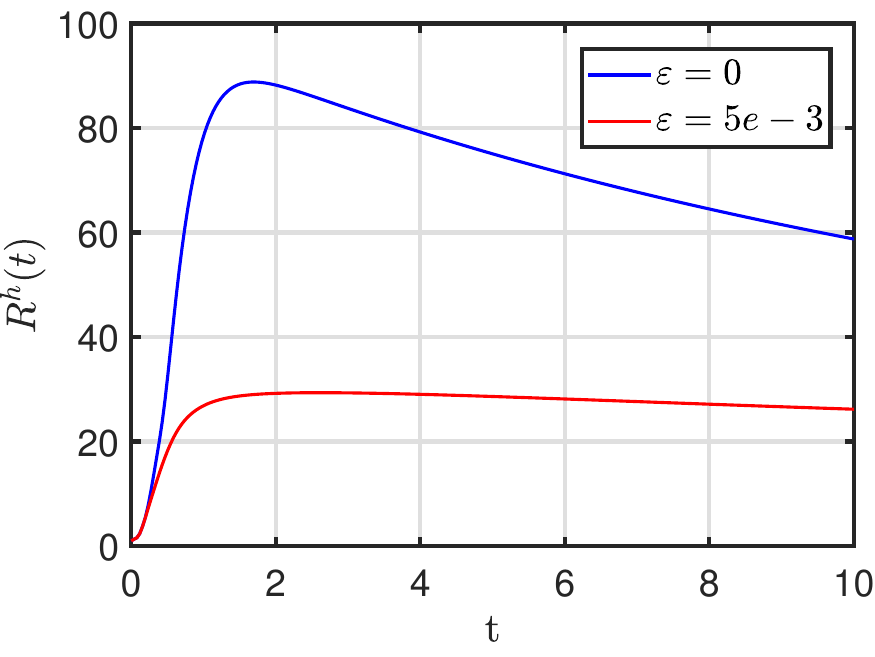}\hspace{0.5cm}
\includegraphics[width=0.45\textwidth]{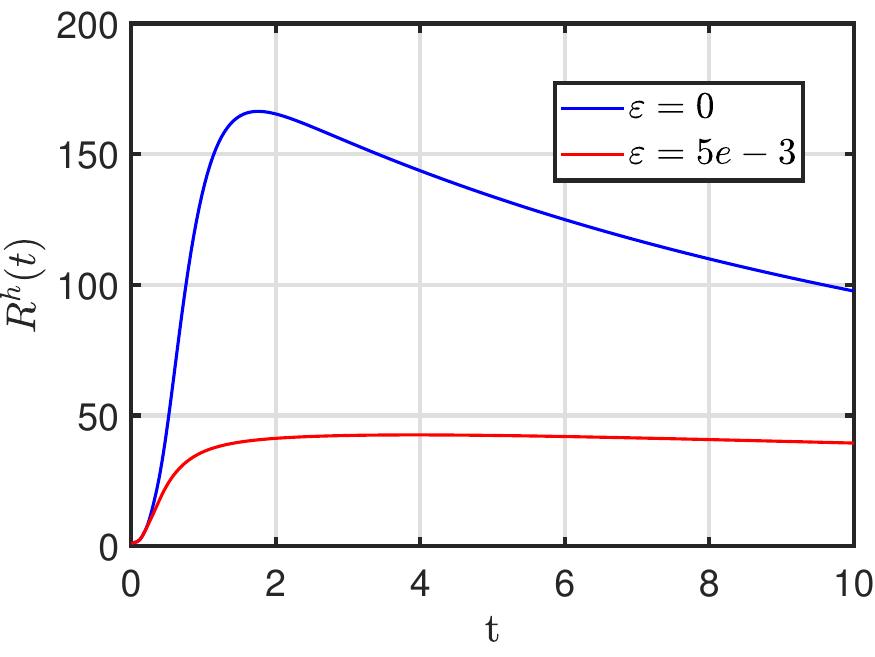}
\caption{Time evolution of the mesh ratio $R^h(t)$ for the four cases of $\beta = 0.12, 0.15, 0.18, 0.2$. The initial data $\vec X^0(\rho) = (\cos(\pi\rho/2), \sin(\pi\rho/2))$, and the parameters are selected as $\eta = 100$, $\sigma = -0.6$, $J = 65$, $\ttau = 5/128$. }
\label{fig:7}
\end{figure}

\textbf{Example 4} (Equilibrium state \& Pinch-off) 
In this example, our focus is primarily on the intrinsic mechanisms involved in the evolution of the thin film and the shape it settles into when it reaches a stable state.
In Figure \ref{fig:8}, we depict the impact of varying $\sigma$ values—specifically, $\sigma = 0.6$, $0$, and $-0.6$—on the equilibrium morphology. The initial configuration is set as a semicircle, defined by $\vec X^0 = (4 + \cos(\pi\rho), \sin(\pi\rho))$. In Figure \ref{fig:9}, we present the evolution curve leading to equilibrium and the associated axisymmetric surface for the structure-preserving method.
In Figure \ref{fig:9}, we also simulate the hole shrinkage effect of SSD, and find that the central hole in the film gradually becomes smaller over time. 
Finally, we study the pinch-off effect that occurs during the evolution of the film. Figures \ref{fig:10}-\ref{fig:11} show the results of pinch-off occurring at two different boundaries, with the initial data $\vec X^0 = (20 + 8\cos(\pi\rho), 0.14\sin(\pi\rho))$ and $\vec X^0 = (6\cos(\pi\rho/2), 0.2\sin(\pi\rho/2))$. We observe that this effect happens when the film becomes very long and flat.

\begin{figure}[!htp]
\centering
\includegraphics[width=0.3\textwidth]{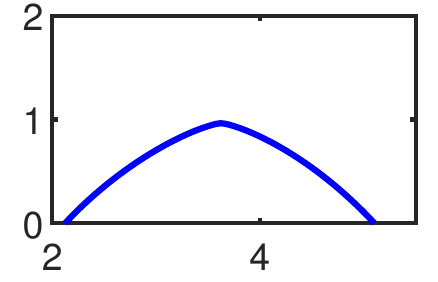}\hspace{0.5cm}
\includegraphics[width=0.3\textwidth]{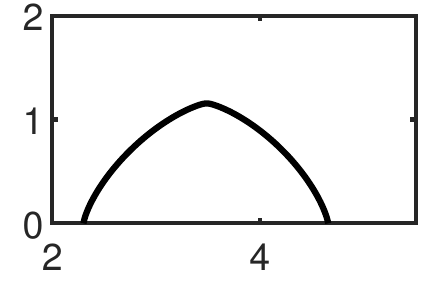}\hspace{0.5cm}
\includegraphics[width=0.3\textwidth]{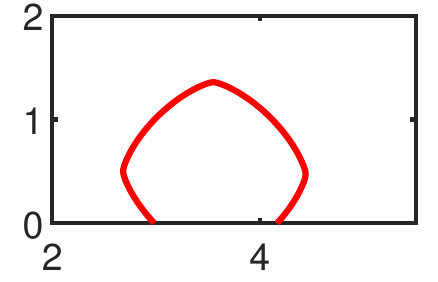}\hspace{0.5cm}
\includegraphics[width=0.3\textwidth]{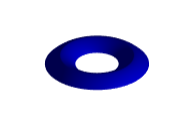}\hspace{0.5cm}
\includegraphics[width=0.3\textwidth]{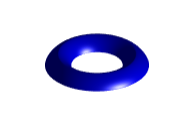}\hspace{0.5cm}
\includegraphics[width=0.3\textwidth]{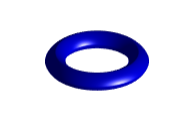}
\caption{On the upper panel we present the generation curves of the equilibrium state for $\sigma = 0.6, 0, -0.6$. On the lower panel, we present the corresponding axisymmetric surface. The initial data $\vec X^0(\rho) = (4 + \cos(\pi\rho), \sin(\pi\rho))$, and the parameters are selected as $\eta = 100$, $\sigma = -0.6$, $J = 100$, $\ttau = 1/100$, $\beta = 0.07$. }
\label{fig:8}
\end{figure}

\begin{figure}[!htp]
\centering
\includegraphics[width=0.6\textwidth]{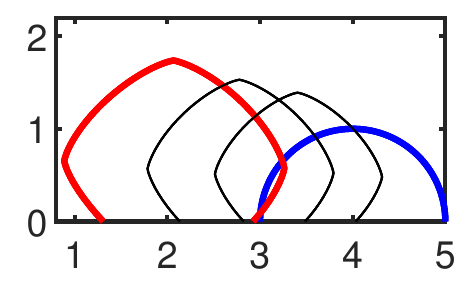}\hspace{10cm}
\includegraphics[width=0.24\textwidth]{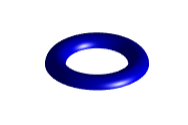}
\includegraphics[width=0.24\textwidth]{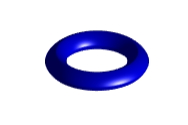}
\includegraphics[width=0.24\textwidth]{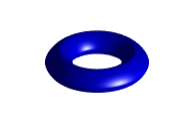}
\includegraphics[width=0.24\textwidth]{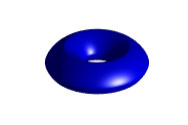}
\caption{On the upper panel we present the generating curves $\Gamma^m$ at $t = 0, 2.5, 6, 8$. On the lower panel, we show the corresponding axisymmetric surface $S^m$ generated by $\Gamma^m$. The initial data $\vec X^0(\rho) = (4 + \cos(\pi\rho), \sin(\pi\rho))$. Here $J = 100$, $\ttau = 1/50$, $\sigma = -0.6$, $\varepsilon = 0.001$. }
\label{fig:9}
\end{figure}

\begin{figure}[!htp]
\centering
\includegraphics[width=0.3\textwidth]{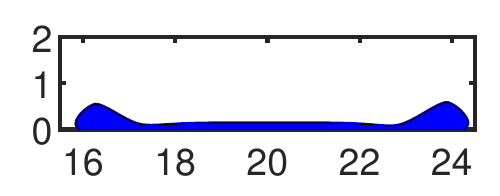}\hspace{0.5cm}
\includegraphics[width=0.3\textwidth]{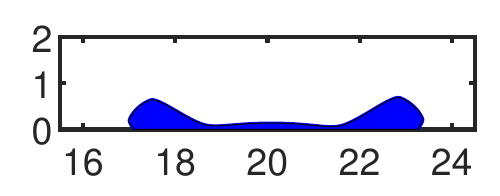}\hspace{0.5cm}
\includegraphics[width=0.3\textwidth]{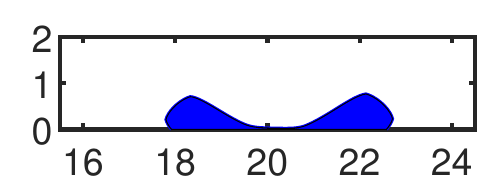}\hspace{0.5cm}
\includegraphics[width=0.3\textwidth]{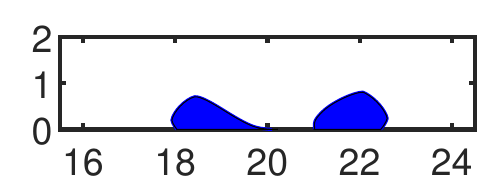}\hspace{0.5cm}
\includegraphics[width=0.3\textwidth]{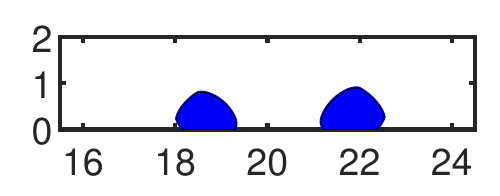}\hspace{0.5cm}
\includegraphics[width=0.3\textwidth]{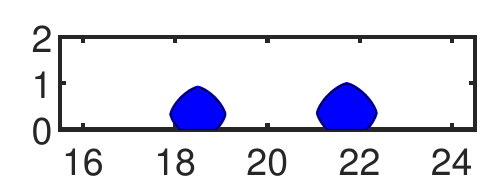}\hspace{0.5cm}
\includegraphics[width=0.48\textwidth]{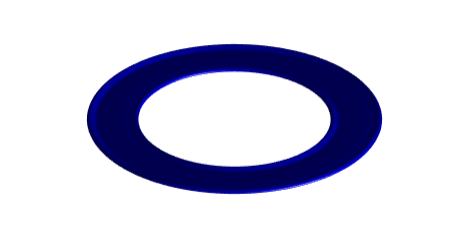}
\includegraphics[width=0.48\textwidth]{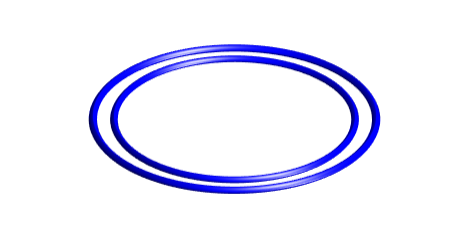}
\caption{On the upper panel we present the generating curves $\Gamma^m$ at $t = 0.2, 0.4, 0.64, 0.7, 0.74, 2.64$. On the lower panel, we show the corresponding axisymmetric surface $S^m$ generated by $\Gamma^m$ at $t = 0.2, 2.64$. The initial data $\vec X^0(\rho) = (20 + 8\cos(\pi\rho), 0.14\sin(\pi\rho))$. Here $J = 100$, $\ttau = 1/50$, $\sigma = -0.6$, $\varepsilon = 0.001$. $\beta = 0.07$. }
\label{fig:10}
\end{figure}

\begin{figure}[!htp]
\centering
\includegraphics[width=0.3\textwidth]{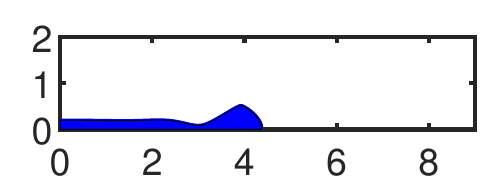}\hspace{0.5cm}
\includegraphics[width=0.3\textwidth]{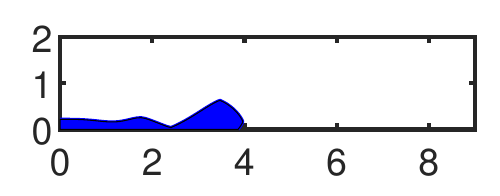}\hspace{0.5cm}
\includegraphics[width=0.3\textwidth]{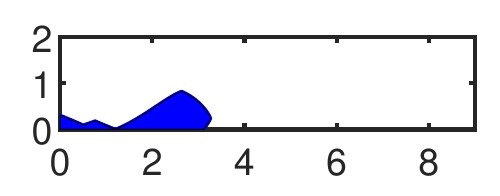}\hspace{0.5cm}
\includegraphics[width=0.3\textwidth]{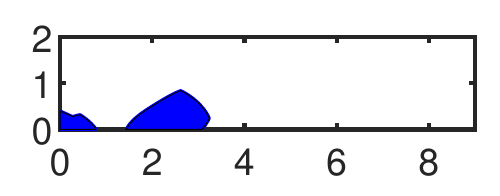}\hspace{0.5cm}
\includegraphics[width=0.3\textwidth]{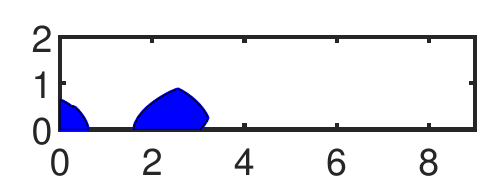}\hspace{0.5cm}
\includegraphics[width=0.3\textwidth]{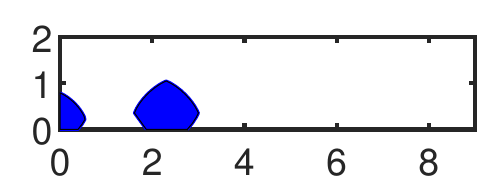}\hspace{0.9cm}
\includegraphics[width=1\textwidth]{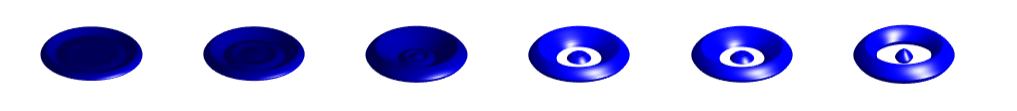}
\caption{On the upper panel we present the generating curves $\Gamma^m$ at $t = 0.05, 0.1, 0.27, 0.28, 0.29, 0.34$. On the lower panel, we show the corresponding axisymmetric surface $S^m$ generated by $\Gamma^m$. The initial data $\vec X^0(\rho) = (6\cos(\pi\rho/2), 0.2\sin(\pi\rho/2))$. Here $J = 50$, $\ttau = 1/200$, $\sigma = -0.6$, $\varepsilon = 0.001$. $\beta = 0.1$. }
\label{fig:11}
\end{figure}

\section{Conclusions}\label{sec7}
In this work, 
through the application of thermodynamic variation principles for a new-defined regularized total free energy, we derive a sharp-interface model to capture the dynamics of axisymmtric SSD with strong anisotropies.
Furthermore, we develop structure-preserving parametric finite element approximation for the sharp-interface model, ensuring both volume conservation and energy stability.
The main motivation for constructing the regularized system in this work is the inclusion of the Willmore regularization term, which can ensure the well-posedness of the model.
By constructing two novel geometric relationships, we establish an equivalent regularized sharp-interface model and further develop a structure-preserving numerical scheme tailored to this new model, which fills a gap in the existing theoretical framework. 
A large number of numerical experiments demonstrate the accuracy and structure-preserving properties of the numerical scheme.
Most importantly, compared to the system without the regularization term, numerical simulations show that the scheme maintains good mesh quality throughout the evolution process. 

%\section{Acknowledgments}\label{sec8}
\section*{Acknowledgments}
This work has been funded by the National Natural Science Foundation of China [Nos. 11801527, U23A2065].

\bibliographystyle{elsarticle-num}
\bibliography{thebib}
\end{document}